\documentclass[12pt]{article}
\usepackage[margin=1.5in]{geometry} 
\usepackage{amsmath,amsthm,amssymb,amsfonts}
\usepackage{changepage}
\usepackage{tikz}
\usepackage[shortlabels]{enumitem}
\usetikzlibrary{matrix}

\newcommand{\Z}{\mathbb{Z}}
\newcommand{\N}{\mathbb{N}}

\newcommand{\F}{\mathcal{F}}

\newcommand{\asdim}{\textnormal{asdim}}
\newcommand{\asi}{\textnormal{asi}}

\newcommand{\res}{\upharpoonright}

\newtheorem{theorem}{Theorem}

\newtheorem{corollary}{Corollary}
\newtheorem{lemma}{Lemma} 
\newtheorem{prop}{Proposition}  
\newtheorem{prob}{Problem} 
 
\begin{document}
 
\title{Borel Edge Colorings for Finite Dimensional Groups}
\author{Felix Weilacher}
\maketitle

\begin{abstract}
\noindent We study the potential of Borel asymptotic dimension, a tool introduced recently in \cite{dim}, to help produce Borel edge colorings of Schreier graphs generated by Borel group actions. We find that it allows us to recover the classical bound of Vizing in certain cases, and also use it to exactly determine the Borel edge chromatic number for free actions of abelian groups. 
\end{abstract}


\section{Introduction}\label{sec:intro}


In a recent paper \cite{dim}, Conley et al. introduced a new tool, Borel asymptotic dimension, to the study of Borel combinatorics. We will define this notion in Section \ref{sec:dim}, but for now we emphasize that they exhibited several useful consequences which follow when a locally finite Borel graph has finite Borel asymptotic dimension, including hyperfiniteness for its connectedness relation, and Borel vertex colorings using relatively few colors. 

The aim of this paper is to add to this list by showing applications of finite Borel asymptotic dimension to edge colorings.

Let $X$ be a standard Borel space and $G \subset X \times X$ a Borel graph on $X$. An \textit{edge coloring} of $G$ is a function $c:G \rightarrow Y$ which sends adjacent edges in $G$, that is, edges sharing a vertex, to distinct elements of $Y$, or ``colors''. Such a $c$ is called a \textit{$k$-coloring} if $|Y| = k$. The \textit{edge chromatic number} of $G$, denoted $\chi'(G)$, is the smallest cardinal $k$ such that $G$ admits a $k$-coloring. The \textit{Borel edge-chromatic number} of $G$, denoted $\chi_B'(G)$, is the smallest cardinal $k$ such that $G$ admits a Borel $k$-coloring. See \cite{KM} for a survey of these and related notions.

The Borel graphs of interest to us in this paper will be those generated by Borel actions of finitely generated groups. A \textit{marked group} is a pair $(\Gamma,S)$ where $\Gamma$ is a group and $S$ is a finite symmetric generating set for $\Gamma$ not containing 1. We will sometimes omit the $S$ if it will not cause confusion. Let $X$ be a standard Borel space and $a:\Gamma \curvearrowright X$ a free Borel action. Let $G(a,S)$ denote the \textit{Schreier graph} generated by this action. This is the graph defined by setting $x,y \in X$ adjacent if and only if $\gamma \cdot_a x = y$ for some $\gamma \in S$.

Note that all the notions in the above paragraph make sense even if $S$ does not generate the entire group $\Gamma$. We shall sometimes use them in this case.

A classical theorem of Vizing states that if $G$ is a graph with maximum degree $d$, then $\chi'(G) \leq d+1$. Of natural interest is the extent to which this bound continues to hold in the Borel setting. Greb\'{i}k and Pikhurko have shown that this bound holds in the measurable setting when $G$ is measure preserving \cite{GP}, but on the other hand Marks has shown that it fails in the general Borel setting, even for acyclic $G$ \cite{M15}.

As was previously promised, in section \ref{sec:dim} we shall see a precise definition of Borel asymptotic dimension for graphs and Borel actions. Actually, of more direct use will be a variant of this number, also from \cite{dim}, called Borel asymptotic separation index. Our main result in Section \ref{sec:d+1} will be that the Vizing's bound holds in the Borel context for graphs generated by certain group actions when this index is 1.

\begin{theorem}\label{th:d+1}
Let $a:\Gamma \curvearrowright X$ be a free Borel action of a marked group $(\Gamma,S)$ on a standard Borel space $X$ with Borel asymptotic separation index 1, and such that none of the elements of $S$ have odd order. Then $\chi_{B}'(G(a,S)) \leq |S|+1$.
\end{theorem}

We emphasize that the condition on the generators in Theorem \ref{th:d+1} implies that $\chi'(G(a,S)) = |S|$. Nevertheless, we will see in Section \ref{sec:Z^d} that the ``$+1$'' in Theorem \ref{th:d+1} cannot always be removed.

In \cite{dim}, it is shown that finite Borel asymptotic dimension implies a Borel asymptotic separation index of 1. Several quite general classes of groups, including those with a polynomial growth rate, are also shown to always have finite Borel asymptotic dimension for their free Borel actions. The work of that paper therefore provides us with many groups to which Theorem \ref{th:d+1} can be applied.

We will also see in Section \ref{sec:d+1} that Theorem \ref{th:d+1} can be applied to finite extensions of marked groups which satisfy its constraints.

Finally, we will note in Section \ref{sec:d+1} that Theorem \ref{th:d+1} implies that the Vizing bound holds in certain Baire measurable or measurable settings.

It follows from the third most recent paragraph that free actions of abelian groups always have finite Borel asymptotic dimension, although this was known earlier from work of Gao et al. \cite{toast}. In Section \ref{sec:Z^d}, we will see how a more specific analysis allows us to improve on Theorem \ref{th:d+1} for these groups by exactly determining their edge chromatic numbers.

\begin{theorem}\label{th:Z^d}
Let $(\Gamma,S)$ be a marked group with $\Gamma = \Delta \times \Z^d$ abelian, where $\Delta$ is the torsion part of $\Gamma$ and $d \in \omega$ is its free rank. Let $a:\Gamma \curvearrowright X$ be the action of $\Gamma$ on the free part of its Bernoulli shift, and $G = G(a,S)$
\begin{enumerate}
    \item If $d = 0$, $\chi'(G) = \chi_B'(G) = |S|$ if and only if $\Delta$ has even order. (and $\chi'(G) = \chi_B'(G) = |S|+1$ otherwise.)
    \item If $d = 1$, $\chi'(G) = |S|$, and $\chi_B'(G) = |S|$ if and only if $\Delta$ has even order (and $\chi_B'(G) = |S|+1$ otherwise.)
    \item If $d \geq 2$, $\chi'(G) = \chi_B'(G) = |S|$.
\end{enumerate}
\end{theorem}

The point of using the Bernoulli shift in the above statement is that arbitrary Borel free actions of $\Gamma$ always admit Borel $\Gamma$-equivariant embeddings to the free part of the Bernoulli shift \cite{JKL}, and so the chromatic numbers for the free part of the Bernoulli shift are the supremums of the chromatic numbers of arbitrary free actions of $\Gamma$.

Note that the statements in parentheses follow from Vizing's Theorem/ Theorem \ref{th:d+1} (or rather the aforementioned generalization of the theorem to finite extensions)  and the fact that when $\Gamma$ is finite, its discrete and Borel combinatorics coincide. 

During the preparation of this project, we became aware that this result was recently arrived at two other times independently in the special case $\Gamma = \Z^d$ and $S =$ the standard generating set. \cite{GR},\cite{BHT} (the latter for $d = 2$). The key ideas from the proof given in \cite{GR} appear also in this work, but some additional ideas are needed to extend the result to arbitrary abelian groups and arbitrary generating sets. Additionally, we learned via personal communication that the authors of \cite{GR} were aware of Theorem \ref{th:d+1}. Our paper provides the first write up of this result. 

On that note, we mention one interesting difference between our approach and those in \cite{GR} and \cite{BHT}. Throughout the paper, we use having a Borel asymptotic separation index of 1 as our key background assumption, where most existing work has used the existence of what is sometimes called a ``toast'' structure, following \cite{toast}, which is a particularity nice witness to hyperfiniteness. The existence of toast can easily be seen to imply asymptotic separation index 1, but it is unclear whether the converse holds. See Problem \ref{prob:hyperfinite}.


\section{Borel asymptotic dimension and asymptotic separation index}\label{sec:dim}


In this section we will review the relevant definitions and facts from \cite{dim} regarding Borel asymptotic dimension and asymptotic separation index. 

We start with the discrete versions of these notions: Let $G$ be a locally finite graph on a set $X$, and $N$ a positive integer. Let $G^{\leq N}$ denote the \textit{distance $N$-graph} of $G$. This is the graph which sets two vertices adjacent if and only if there is a path from one to the other of length at most $N$.

The \textit{asymptotic separation index} of $G$, denoted $\asi(G)$ is the infimum over $s \in \omega$ such that the following holds: For every positive integer $N$, $X$ can be partitioned into $s+1$ sets, say $U_0,\ldots,U_s$, such that for each $0 \leq i \leq s$, the induced graph $G^{\leq N} \res U_i$ has only finite connected components. The \textit{asymptotic dimension} of $G$, denoted $\asdim(G)$, is defined similarly, except now the connected components of each $G^{\leq N} \res U_i$ are required to have a uniform (finite) bound on their diameter. 

 Let $G$ now be a Borel graph on a standard Borel space $X$. The \textit{Borel asymptotic separation index} and \textit{Borel asymptotic dimension} of $G$, denoted $\asi_B(G)$ and $\asdim_B(G)$ respectively, are defined exactly as their discrete counterparts were, except that now all the $U_i$'s are required to be Borel sets.
 
 Now let $\Gamma$ be a finitely generated group, and $a:\Gamma \curvearrowright X$ a free Borel action. It is easy to verify that none of the above numbers for the graph $G(a,S)$ depend on the choice of generating set $S$ for $\Gamma$. We can thus talk about the (Borel) asymptotic separation index and (Borel) asymptotic dimension of the action $a$ itself.
 
 We now list some facts. If $G$ has only finite connected components, then $\asi_B(G)$ is trivially 0 and its Borel and discrete combinatorics coincide, so until Section \ref{sec:Z^d} assume this is not the case. The following facts, which were referenced in Section \ref{sec:intro}, can be used to establish that Borel actions of finitely generated free abelian groups have Borel asymptotic separation index 1, and are also relevant to the discussion more broadly.
 
 \begin{theorem}[\cite{dim}]\label{th:poly}
 Let $\Gamma$ be a finitely generated group with polynomial growth rate for its Cayley graph(s). Then any free Borel action of $\Gamma$ has finite Borel asymptotic dimension. 
 \end{theorem}
 
 \begin{theorem}[\cite{dim}]\label{th:asi=1}
 Let $G$ be a locally finite Borel graph with $\asdim_B(G) < \infty$. Then $\asi_B(G) = 1$.
 \end{theorem}
 
 These next facts will sometimes allow us to recover Vizing's bound in the Baire measurable and measurable contexts, even for groups for which $\asi_B \neq 1$.
 
\begin{theorem}[\cite{CM},\cite{dim}]\label{th:category}
Let $G$ be a locally finite Borel graph on a Polish space $X$. Then there is a Borel $G$-invariant comeager set $X' \subset X$ such that $\asi_B(G \res X') = 1$.
\end{theorem}

\begin{theorem}[\cite{CM}]\label{th:measure}
Let $G$ be a locally finite Borel graph on a standard Borel probability measure space $(X,\mu)$ such that the connectedness equivalence relation of $G$ is hyperfinite. Then there is a Borel $G$-invariant $\mu$-conull set $X' \subset X$ such that $\asi_B(G \res X') = 1$.
\end{theorem}
 
Finally, we mention a problem in this context which appears to be open, and which was mentioned at the end of Section \ref{sec:intro}.

\begin{prob}\label{prob:hyperfinite}
Let $G$ be a locally finite Borel graph with $\asi_B(G) = 1$. Is The connectedness equivalence relation of $G$ hyperfinite?
\end{prob}

In \cite{dim}, this is shown under the stronger assumption $\asdim_B(G) < \infty$.


\section{Degree plus one colorings}\label{sec:d+1}


In this section we prove Theorem \ref{th:d+1} and give some examples where it cannot be improved. 

\subsection{Proof of Theorem \ref{th:d+1}}

Let $G$ be a graph on a set $X$. An \textit{injective $G$-ray} is an injective infinite sequence $x_0,x_1,\ldots \in X$ such that $(x_i,x_{i+1}) \in G$ for all $i \in \omega$. Given a set $U \subset X$ and an $n \in \omega$, let us write $B(U,n)$ for the ball of radius $n$ around $U$. That is, the set of all points $x \in X$ whose path distance to $U$ is less than or equal to $n$. For $m \in \omega$, let us write $A(U,n,m) = B(U,m)-B(U,n)$. One can check that if $G$ and $U$ are Borel, these sets are always Borel.

The following two lemmas together explain how having a Borel asymptotic separation index of 1 will help us to construct edge colorings.

\begin{lemma}\label{lem:rainbow}
Suppose $G$ is a locally finite Borel graph on a space $X$ with $\asi_B(G) = 1$, and $d \in \omega$. Then we can find pairwise disjoint Borel sets $V_1,\ldots,V_d \subset X$ such that any injective $G$-ray contains infinitely many edges within each $V_i$.
\end{lemma}

\begin{proof}
Since the tail of an injective $G$-ray is still an injective $G$-ray, it suffices to prove this with ``infinitely many edges'' replaced by ``an edge''. 

Let $X = U_0 \sqcup U_1$ be a Borel partition of $X$ witnessing $\asi_B(G) = 1$ for $N = 4d+1$. That is, such that the graphs $G^{\leq N} \res U_0$ and $G^{\leq N} \res U_1$ both have only finite connected components. Set $V_i = A(U_0,2(i-1),2i)$ for each $1 \leq i \leq d$. Also set $W = X \setminus B(U_0,2d) \subset U_1$. As was mentioned above, these are all still Borel. Let us now show these $V_i$'s work.

Let $x = (x_n)$ be an injective $G$-ray. We first show that $x$ contains some point from $W$. Suppose not. Then for each $n$, there is a $y_n \in U_0$ whose path distance to $x_n$ is less than or equal to $2d$. Now for each $n$, going from $y_n$ to $x_n$, then to $x_{n+1}$, then to $y_{n+1}$, produces a path from $y_n$ to $y_{n+1}$ of length at most $2d+1+2d=N$, so all the $y_n$'s are in the same $G^{\leq N} \res U_0$-connected component. By definition of $U_0$, then, the set $\{y_n \mid n \in \omega\}$ is finite. But now $x$ gives infinitely many points within distance $2d$ of this set, contradicting local finiteness.

Since $W \subset U_1$, the above shows in particular that $x$ contains some point from $U_1$. Similarly, $x$ must contain some point from $U_0$. Let $x_n \in W$ and $x_m \in U_0$. We can assume $m < n$. Now the path distance from $x_i$ to $U_0$ changes by at most 1 each time $i$ is incremented by 1. Thus, since $x_m$ has distance 0 from $U_0$ and $x_n$ has distance greater than $2d$ from $U_0$, for each $1 \leq i \leq d$, there must be some $m < j < n$ such that $x_{j}$ and $x_{j+1}$ have distance $2i-1$ and $2i$ respectively from $U_0$. Then $(x_j,x_{j+1})$ is an edge of $x$ contained in $V_i$, as desired.

\end{proof}

\begin{lemma}\label{lem:Z}
Let $a:\Z \curvearrowright X$ be a free Borel action of $\Z$, and $S$ the usual generating set for $\Z$. Suppose $V \subset X$ is a Borel set with the property that every injective $G(a,S)$-ray contains infinitely many edges within $V$. Then there is a Borel 3-edge coloring of $G(a,S)$, say with the colors $1,2$, and $3$, such that the color 3 only occurs on edges contained within $V$.
\end{lemma}

\begin{proof}
A well known theorem from \cite{KST} (Proposition 4.2) states that for locally finite Borel graphs, it is always possible to find Borel maximal independent sets. Apply this to the edge graph of $G(a,S) \res V$ to get a maximal set of edges $A$ contained within $V$ such that no two edges in $A$ share a vertex. Give all the edges in $A$ the color 3.

By maximality and the defining condition for the set $V$, these edges occur infinitely often in both directions along each $\Z$-orbit. Thus, the connected components of $G(a,S) \setminus A$ are all simply finite paths. These can of course be 2-colored with the colors 1 and 2, and since they are all finite, this can be done in a Borel fashion. (For example, by fixing a Borel linear order on the space of 2-edge colored finite subgraphs of $G(a,S)$, and picking the least 2-coloring in this order for each path above.)
\end{proof}

The proof of Theorem \ref{th:d+1} is now easy:

\begin{proof}
By hypothesis, we may write $S = S_0 \sqcup S_1$, where $S_0$ contains all of our generators of infinite order, and $S_1$ contains all of our elements of even (finite) order. For each $\gamma \in S_1$, the connected components of the graph $G(a,\{\gamma^{\pm 1} \})$ are all even (finite) cycles, or just single edges if $\gamma$ has order 2. Thus they may be 2-colored in a Borel fashion as in the previous proof (or just 1-colored if $\gamma$ has order 2). Thus we may Borel $|S_1|$-color the edges in $G(a,S_1)$.

It now suffices to Borel $|S_0|+1$-color the edges in $G(a,S_0)$ using a disjoint set of colors. Let $S_0 = \{\gamma_1^{\pm 1},\ldots,\gamma_d^{\pm 1}\}$, so that $|S_0| = 2d$. We will use the color set $\{1,\ldots,2d+1\}$. Let $V_1,\ldots,V_d \subset X$ be the sets from Lemma \ref{lem:rainbow}. Since $\langle \gamma_i \rangle \cong \Z$ for each $i$, by Lemma \ref{lem:Z} we may for each $i$ color the edges in $G(a,\{ \gamma_i^{\pm 1} \})$ using the colors $2i-1$, $2i$, and $2d+1$ such that the color $2d+1$ only appears on edges contained within $V_i$. Since the $V_i$'s are pairwise disjoint, this does not cause any color conflicts, so we are done.
\end{proof}

We now wish to see that we can extend Theorem \ref{th:d+1} to finite extensions. To do this, it will be helpful to slightly expand our notion of graph to allow for multiple edges.

Let $\Gamma$ be a group and $S \subset \Gamma$ a finite symmetric multi-set of generators for it not containing 1. That is, the elements of $S$ can appear with some multiplicity. With $a:\Gamma \curvearrowright X$ a Borel free action as before, let $G(a,S)$ denote the Borel-multigraph where now points $x$ and $y$ are connected with one edge for each $\gamma \in S$ such that $\gamma \cdot x = y$. Edge colorings can be defined as before. Of course, if $x$ and $y$ have multiple edges between them, those all need different colors in an edge coloring.

Observe that Theorem \ref{th:d+1} still holds, with the same proof, if $S$ is allowed to be a multi-set. This allows us to prove:

\begin{corollary}\label{cor:d+1}
Let $a:\Gamma \curvearrowright X$ be a free Borel action of a marked group $(\Gamma,S)$ on a standard Borel space $X$ with Borel asymptotic separation index 1. Suppose $\Delta \leq \Gamma$ is a finite normal subgroup of $\Gamma$, and that for each $\gamma \in S \setminus \Delta$, the image of $\gamma$ in the quotient $\Gamma/\Delta$ does not have odd order. Then $\chi_B'(G(a,S)) \leq |S|+1$.
\end{corollary}

\begin{proof}
Write $S = S_0 \sqcup S_1$, where $S_1 = S \cap \Delta$. Consider first the graph $G(a,S_1)$. Its connected components are all contained within $\Delta$-orbits, and therefore finite. They are also $|S_1|$-regular, so by the classical Vizing's Theorem and the argument from the final paragraph of Lemma \ref{lem:Z}, we may find a Borel $(|S_1|+1)$-edge coloring of this graph, say using the color set $C$.

Now let $\overline{X}$ be the Borel space of $\Delta$-orbits. Then $a$ induces a free Borel action, say $\overline{a}$, of $\Gamma/\Delta$ on $\overline{X}$. Let $\overline{S_0}$ be the multi-set which results from considering the images of the elements of $S_0$ in $\Gamma/\Delta$. By the last paragraph before the statement of the corollary, we may find a Borel $(|S_0|+1)$-edge coloring using a disjoint set of colors from $C$, say $C'$, of $G(\overline{a},\overline{S_0})$. Now, given an edge $(x,\gamma\cdot x)$ in $G(a,S_0)$, give it the color received by the edge $(\Delta \cdot x, \Delta \gamma \cdot x)$ in $G(\overline{a},\overline{S_0})$ corresponding to $\gamma$. This results in a Borel $(|S_0|+1)$-edge coloring of $G(a,S_0)$ since whenever $x,y$ are distinct vertices in the same $\Delta$-orbit, the edges $(x,\gamma \cdot x)$ and $(y,\gamma \cdot y)$ share no vertices. 

Amalgamating these two colorings gives us a Borel $(|S|+2)$-edge coloring of $G(a,S)$, but we can remove a color: Fix a color $c$ from $C$. For each $\Delta$-orbit, say $E \in \overline{X}$, $E$ has degree $|S_0|$ in the graph $G(\overline{a},\overline{S_0})$, and therefore is incident to only $|S_0|$ colors in our coloring of that graph. By our construction, it is still incident to only those colors in the resulting coloring of $G(a,S_0)$. Thus, we may pick a color $c'$ from $C'$ which $E$ does not meet. Then we may change the color of any edge within $E$ colored $c$ to $c'$. This process maintains Borel-ness, and results in an edge coloring not using $c$, hence a ($|S|+1$)-edge coloring.
\end{proof}

The following variant of Corollary \ref{cor:d+1} will be used in Section \ref{sec:Z^d}:

\begin{corollary}\label{cor:d}
Using the notation from the proof of Corollary \ref{cor:d+1}, \ suppose $\chi'(G(a,S_1)) = |S_1|$ (That is, the Cayley graph of $(\Delta,S_1)$ is degree-edge colorable) and $S_1 \neq \emptyset$. Then $\chi_B'(G(a,S)) = |S|$.
\end{corollary}

\begin{proof}
Repeat the proof of Corollary 1, but now where the color set $C$ has size $|S_1|$. The condition $S_1 \neq \emptyset$ is needed for the final step in the proof where we fix a color from $C$.
\end{proof}

Corollary \ref{cor:d+1} also generalizes a previous result of this author, which was proved using similar ideas \cite{W21}: 

\begin{corollary}\label{cor:2end}
Let $a:\Gamma \curvearrowright X$ be a free Borel action of a marked group $(\Gamma,S)$ with two ends. Then $\chi_B'(G(a,S)) \leq |S|+1$.
\end{corollary}
\begin{proof}
Two ended groups are always finite extensions of $\Z$ or the infinite dihedral group $D_\infty$, neither of which have any odd order elements. Furthermore a Borel asymptotic separation index of 1 is automatic by Theorem \ref{th:poly} since $\Gamma$ has linear growth rate, though this also follows from earlier work of Miller \cite{2end}.
\end{proof}

Similarly, Corollary \ref{cor:d+1} implies Vizing's bound for any action of an abelian group:

\begin{corollary}
Let $a:\Gamma \curvearrowright X$ be a free Borel action of an abelian marked group $(\Gamma,S)$. Then $\chi_B'(G(a,S)) \leq |S|+1$.
\end{corollary}
\begin{proof}
By the fundamental theorem of finitely generated abelian groups, $\Gamma$ is a finite extension of some $\Z^d$. This implies it has polynomial (degree $d$) growth rate, and $\Z^d$ does not have any odd order elements.
\end{proof}

Note that the combination of Corollary \ref{cor:2end} and Theorem \ref{th:Z^d} actually render this Corollary redundant. It is still nice to state here, though, as Theorem \ref{th:Z^d} will require somewhat more work. 

We now pause briefly to address the measurable and Baire measurable situations. If $G$ is a Borel graph on a Polish space $X$, we denote by $\chi_{BM}'(G)$ the minimum of $\chi_B'(G \res X')$ as $X'$ ranges over all Borel comeager $G$-invariant subsets of $X$. If $X$ is instead equipped with a Borel probability measure $\mu$, $\chi_\mu'(G)$ is defined similarly. Note that these numbers are both lower bounds for $\chi_B'(G)$. Now, Theorems \ref{th:category} and \ref{th:measure} have the following obvious consequences:

\begin{corollary}
If we replace $\chi_B'$ with $\chi_{BM}'$, Theorem \ref{th:d+1} and Corollary \ref{cor:d+1} hold for Borel actions on Polish spaces without any assumptions on asymptotic separation index. 

Similarly in the measurable setting for actions with hyperfinite orbit equivalence relation. 
\end{corollary}

\subsection{Examples of tightness}

Let us now consider when the ``$+1$'' in Theorem \ref{th:d+1} cannot be removed. It is well known that there are free Borel actions of $\Z$ with its usual generating set which cannot be even Baire measurably 2-edge colored \cite{KST} (page 11, essentially), and so the bound in Theorem \ref{th:d+1} is certainly sometimes tight. It would be reasonable, however, to chalk this up to the fact that 2-colorings are very rigid. For example, the following fact shows that the situation can sometimes be very different for graphs of max degree $>2$:

\begin{theorem}[\cite{KM}]
If $G$ is an acyclic, $d$-regular Borel graph on a Polish space and $d > 2$, then $\chi_{BM}'(G) = d$.
\end{theorem}

Thus it is natural to ask for examples in our situation for larger degrees. The following shows that Theorem \ref{th:d+1} remains tight for graphs of arbitrarily large degree, even if they are also assumed to be bipartite. It also provides one of the directions in the $d=1$ case of Theorem \ref{th:Z^d}.

\begin{prop}\label{prop:Zexample}
Let $(\Gamma,S)$ be a marked group with $\Gamma = \Delta \times \Z$ and $\Delta$ finite of odd order. Let $a:\Gamma \curvearrowright X$ be the usual action of $\Gamma$ on the free part of its Bernoulli shift. Then $G(a,S)$ does not admit a Baire measurable perfect matching (i.e, it does not admit a Borel perfect matching on a $\Gamma$-invariant comeager set). In particular $\chi_{BM}'(G(a,S)) = \chi_B'(G(a,S)) = |S|+1$.
\end{prop}

\begin{proof}
Suppose, to the contrary, there is a Borel $\Gamma$-invariant comeager set $X' \subset X$ and a Borel perfect matching $M \subset (G(a,S) \res X'
)$. 

Start by defining $\overline{a}:\Z = \Gamma/\Delta \curvearrowright \overline{X}$, $S_0$, and $\overline{S_0}$ exactly as in the proof of Corollary \ref{cor:d+1}. Also let $\overline{X'} \subset \overline{X}$ be the set of $\Delta$-orbits contained in $X'$.

Let $\overline{M} \subset G(\overline{a},\overline{S_0})$ be the submultigraph which includes an edge between distinct $\Delta$ orbits $y$ and $y'$ for each edge between $y$ and $y'$ in $M$. Since $\Delta$ has odd order, each orbit $y \in \overline{X'}$ must be contained in an odd number of edges in $\overline{M}$.

Let $n \in \omega$ be large enough so that $S \subset \Delta \times (-n,n)$. For any edge $(y,m \cdot y) \in G(\overline{a},\overline{S_0})$ with $m > 0$, define $y$ to be the \textit{left endpoint} of that edge and $m \cdot y$ to be the \textit{right endpoint}. For $y \in \overline{X'}$, define $F(y)$ to be the number of edges in $\overline{M}$ with a right endpoint in $[0,n] \cdot y$ and a left endpoint not in $[0,n] \cdot y$. Put another way, this is the number of edges in $\overline{M}$ which leave the set $[0,n] \cdot y$ to the left.

Let us compare $F(y)$ and $F(1 \cdot y)$. For an edge of $\overline{M}$ to be counted in $F(1 \cdot y)$ but not $F(y)$, it would have to either have $y$ has a left endpoint, or $(n+1) \cdot y$ as a right endpoint. Since $\overline{S_0} \subset (-n,n)$, though, no edge with $(n+1) \cdot y$ as a right endpoint can leave the set $[1,n+1] \cdot y$ to the left, and so this cannot occur. Thus our edge would have to have the form $(y,m \cdot y)$ for some $0 < m < n$, and conversely any edge with this form is counted in $F(1 \cdot y)$ but not $F(y)$. 

Similarly, for an edge of $\overline{M}$ to be counted in $F(y)$ but not $F(1 \cdot y)$, it would have to have $y$ as its right endpoint or $(n+1) \cdot y$ as its left endpoint.  The latter is absurd, though, and so our edge would have to have the form $(m \cdot y, y)$ for some $m < 0$. Once again, conversely any edge with this form is counted in $F(y)$ but not $F(1 \cdot y)$.

Now, $y$ must be the left or right endpoint of each edge in $\overline{M}$ which it belongs to. Let $L$ and $R$ be the number of such edges falling into these respective cases. By the above discussion, we can conclude $F(1 \cdot y) = F(y)  + L - R$. By an earlier comment, though, $L+R$ is odd, so $F(1 \cdot y) = F(y) + 1$ (mod 2). Thus, if we define $f(y)$ to be the mod 2 value of $F(y)$, $f$ is a 2-vertex coloring of $G(\overline{a},\{\pm 1\}) \res \overline{X'}$. It is also clearly Borel since $M$ was. This is well known to be impossible for nonmeager $X'$ by a basic ergodicity argument \cite{KST} (page 11), so we are done. 

Finally, the conclusions about the edge chromatic numbers follow from this and Corollary \ref{cor:2end}, since an $|S|$-edge coloring is a decomposition into perfect matchings.

\end{proof}

Note that this fails if we do not use the convention that $S$ must be finite. Indeed it is not hard to construct Borel perfect matchings for the Borel graph on $X$ above which connects any two points in the same $\Gamma$-orbit. 

\subsection{Open problems}

We end this section by listing interesting open problems related to Theorem \ref{th:d+1}.

First, for Schreier graphs, it seems to be open whether any of the assumptions of the theorem are necessary:

\begin{prob}
Is there any marked group $(\Gamma,S)$ and free Borel action $a:\Gamma \res X$ such that $\chi_B'(G(a,S)) > |S| + 1$?
\end{prob}

Natural candidates to resolve this question are the free groups with their usual generating sets. It follows from the results in \cite{M15} that the Schreier graphs of these marked groups do not necessarily admit Borel pefect matchings, ruling out Borel degree-edge colorings, but not much else seems to be known. (See Problem 5.37 in \cite{KM}.)

It is also unclear whether the presence of a group action is necessary:

\begin{prob}\label{prob:nogroup}
Let $G$ be a locally finite $d$-regular Borel graph with $\asi_B(G) = 1$. Is $\chi_B'(G) \leq d+1$?
\end{prob}

It may be necessary to add an assumption of bipartite-ness to get an affirmative answer to this problem as a replacement to Theorem \ref{th:d+1}'s prohibition of odd-order generators. This restricted version of the problem still seems to be open.\footnote{Addendum: While this paper was under review, an affirmative answer in the bipartite case was indeed found by Matthew Bowen and the author.} Note that by Theorem \ref{th:category}, an affirmative answer to Problem \ref{prob:nogroup} would imply that Vizing's theorem holds in complete generality in the Baire measurable setting.

Given Theorem \ref{th:d+1}, one way to make progress on Problem \ref{prob:nogroup} in the case where $d$ is even would be to find  so-called \textit{Schreier decorations} of $G$ which are Borel. These are orientations of the edges of $G$ along with edge colorings so that each vertex ends up with exactly one in and one out edge of each color, and so they essentially realize $G$ as the Schreier graph of some not-necessarily free group action. Some recent progress has been made in this area: See \cite{T21}, \cite{BHT}. Even more generally, it would be useful to simply find Borel ways of decomposing $G$ into graphs of max degree $2$.

Finally, given Proposition \ref{prop:Zexample} it is interesting to ask whether any groups other than finite extensions of $\Z$ can provide examples where the degree plus one bound is tight. 

\begin{prob}
Is there a marked group $(\Gamma,S)$ and a free Borel action $a:\Gamma \curvearrowright X$ satisfying the hypothesis of Theorem \ref{th:d+1} for which $\chi_B'(G(a,S)) = |S|+1$, and for which $\Gamma$ is not a finite extension of $\Z$? 
\end{prob}

See Question 38 in \cite{BHT} for a similar problem. Note that by Theorem \ref{th:Z^d}, the answer is ``No'' for abelian $\Gamma$.


\section{Degree colorings for abelian groups}\label{sec:Z^d}


In this section we prove Theorem \ref{th:Z^d}. In two situations the theorem asserts the non-existence of degree colorings: In the discrete setting when $d = 0$ and $\Delta$ has odd order, and in the Borel setting when $d = 1$ and $\Delta$ has odd order. We have already mentioned that the latter is covered by Proposition \ref{prop:Zexample}. The former is obvious since a finite graph with an odd number of vertices cannot admit a perfect matching. Therefore in this section we can focus on constructing degree colorings in the remaining cases.

Throughout, let $a$,$X$,$\Gamma$,$S$, $d$, and $\Delta$ be as in the statement of the theorem, unless otherwise stated. Note, though, that nothing in this section is specific to the Bernoulli shift. We only use that $a$ is a free Borel action of $\Gamma$.

\subsection{Free rank 0}

In this subsection we complete the proof of statement 1 of Theorem \ref{th:Z^d} by showing that if $\Gamma = \Delta$ is an even order finite abelian group, then $\chi'(G(a,S)) = |S|$. Note that this will imply the same for $\chi_B'(G(a,S))$, by the argument from the second half of Lemma \ref{lem:Z}. While it seems unlikely that this result is original, our coverage of it will introduce a general technique which will be useful for dealing with the torsion part when $d > 0$. We feel it is also simply nice to include for the sake of completeness. 

It suffices to consider the case when $a$ is the action of $\Delta$ on itself by multiplication. In this case our graph is called the \textit{Cayley} graph of $(\Delta,S)$. 

The following lemma is not specific to the case $d=0$.

\begin{lemma}\label{lem:quotient}
Let $\Delta' \leq \Gamma$ be a finite proper subgroup. Let $\overline{\Gamma} = \Gamma/\Delta'$, and then define $\overline{a}:\overline{\Gamma} \curvearrowright \overline{X}$, $S_0$, and $\overline{S_0}$ as in the proof of Corollary \ref{cor:d+1}. If $\chi'(G(\overline{a},\overline{S_0})) = |{S_0}|$, then $\chi'(G(a,S)) = |S|$. Likewise for the Borel edge chromatic numbers of these graphs.
\end{lemma}

\begin{proof}
Let $S_1 = S \cap \Delta' = S \setminus S_0$. By hypothesis, we may find a $|S_0|$-edge coloring of $G(\overline{a},\overline{S_0})$, say using the set of colors $C$. As in the proof of Corollary \ref{cor:d+1}, we may lift this to a $|S_0|$-edge coloring of $G(a,S_0)$ by giving the edge $(x,\gamma \cdot x)$ for $x \in X$ and $\gamma \in S_0$ whichever color was given to the edge $(\Delta' \cdot x, \gamma \Delta' \cdot x)$ in $G(\overline{a},\overline{S_0})$ corresponding to $\gamma$.

Fix a color $c \in C$, which exists since $\Delta'$ is proper, so $S_0$ is nonempty. Consider the subgraph $G'$ of $G(a,S)$ consisting of $G(a,S_1)$ along with all the edges colored $c$ in the previous step. 

Since, in our coloring of $G(\overline{a},\overline{S_0})$, the set of edges colored $c$ gives a perfect matching, each connected component of our subgraph looks like the following: Two adjacent $\Delta$-orbits, say $y$ and $\gamma \cdot y$ for some $\gamma \in S_0$, with all the internal edges from $S_1$, along with all edges between them of the form $(x,\gamma \cdot x)$ for $x \in y$.

Fix a set $C'$ of $|S_1|+1$ colors disjoint from $C$. We now use $C'$ to edge color the above component. First, by Vizing's theorem, we may use $C'$ to color the edges within $y$. Now, for every such edge, say $(x,x')$, give the edge $(\gamma \cdot x, \gamma \cdot x')$ the same color. Since $\Gamma$ is abelian, this colors all edges within $\gamma \cdot y$ without conflict. Finally, for each $x \in y$, by this construction the set of colors of edges in $y$ meeting $x$ is the same as the set of colors of edges in $\gamma \cdot y$ meeting $\gamma \cdot x$. This set has size $|S_1|$, though, so we still have one free color, and can assign it to the edge $(x,\gamma \cdot x)$, completing our coloring.

Thus we can edge color $G'$ using $C'$, and we already have an edge coloring of $G(a,S) \setminus G'$ using $C \setminus \{c\}$, so unioning these gives a coloring using $|C|'+|C|-1 = |S_1|+1 + |S_0| - 1 = |S|$ colors as desired.

Finally, if our original coloring was Borel, our lift will still be Borel, and then our construction can be done in a Borel fashion since we only need to work with finite components. Again this uses the argument from Lemma \ref{lem:Z}.
\end{proof}

Statement 1 now follows:

\begin{proof}
Since $\Delta$ has even order, we may by the fundemental theorem of finite abelian groups find an index 2 subgroup $\Delta' \leq \Delta$. We now wish to apply Lemma \ref{lem:quotient} with this choice of $\Delta'$. We clearly can, though, since we will have $\overline{\Gamma} = \Z/2$, and so $G(\overline{a},\overline{S_0})$ will just consist of two points with $|S_0|$-many edges between them. (Recall we started with the Cayley graph of $(\Delta,S)$.)
\end{proof}

\subsection{Free rank $\geq 2$}\label{subsec:2}

In this subsection we complete the proof of Statement 3 of Theorem \ref{th:Z^d}. Applying Lemma \ref{lem:quotient} with $\Delta' = \Delta$, we reduce to the case $\Gamma = \Z^d$, but where we now allow multiplicity for the generators in $S$. For the rest of the subsection, assume $(\Gamma,S)$ has this form.

 For motivation, consider Figures \ref{fig:d+1} and \ref{fig:Z^d}. The former shows that we can think about the proof of Theorem \ref{th:d+1} as follows: For each $i$, the orbits of $\gamma_i$ are edge colored using mostly the colors $2i-1$ and $2i$, but with an occasional $2d+1$ thrown in to help align parity.
 
 The latter makes an analogy between this and our upcoming proof of Statement 3. For notational convenience, consider the case where $S = \{\pm e_i \mid 1 \leq i \leq d\}$ is the standard generating set, where $e_i = (0,\ldots,0,1,0,\ldots,0)$, with the $1$ in the $i$-th coordinate. We will still color the orbits of each $e_i$ using mostly the colors $2i-1$ and $2i$, but as before we will need to occasionally swap parity. Thanks to the structure of $\Z^d$, though, when $d \geq 2$ we will be able to do this by ``borrowing'' a color from some $e_j$-edge, $j \neq i$.

\begin{figure}
\centering
\includegraphics[width=0.8\textwidth]{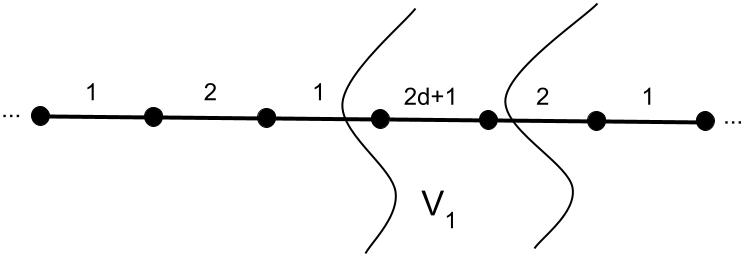}
\caption{\label{fig:d+1} The key idea in the proof of Theorem \ref{th:d+1}. The color $2d+1$ is used (within the region $V_1$) to ``swap parity'' along the orbits of $\gamma_1$.}
\end{figure}

\begin{figure}
\centering
\includegraphics[width=0.8\textwidth]{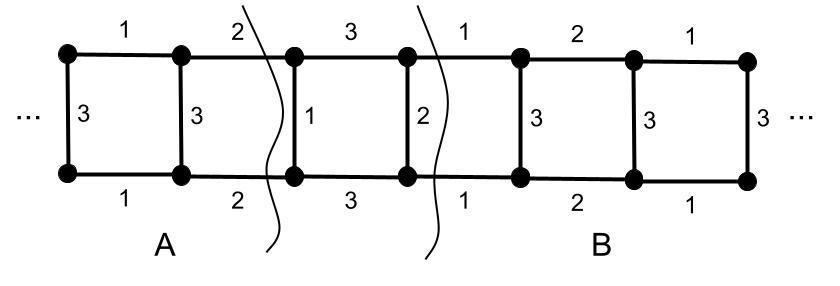}
\caption{\label{fig:Z^d} The key idea in the proof of Theorem 2. Now parity along the horizontal orbits is swapped without the use of an additional color. The regions $A$ and $B$ are labeled so that they can be referenced later.}
\end{figure}

\subsubsection{Borrowing colors}

We will now describe a general setup for using this ``borrowing'' idea. See Figure \ref{fig:general} for a visualization. Let $\gamma_1,\gamma_2 \in S$ such that $\langle \gamma_1,\gamma_2 \rangle \cong \Z^2$. That is, $\gamma_1$ and $\gamma_2$ are not scalar multiples of each other. Work in some finite subset $U$ of some $\Z^d$-orbit, with two disjoint subsets $A,B \subset U$. Let $S' = \{\pm \gamma_1, \pm \gamma_2\}$.

Let $\F$ denote the set of $\gamma_1$-orbits meeting $U$. Call two elements of $\F$ \textit{adjacent} (with respect to $\gamma_2$) if the action of $\gamma_2$ sends one to the other. Note that this makes sense since $\Z^d$ is abelian. Suppose $P$ is a partition of $\F$ into adjacent pairs.  Suppose $e = (x,\gamma_1 \cdot x)$ is a $\gamma_1$-edge in some $f \in \F$. Let $\epsilon \in \{\pm 1\}$ be the unique sign so that $\{f,\epsilon \gamma_2 \cdot f\} \in P$. Then let us call the edge $(\epsilon \gamma_2 \cdot x, \epsilon \gamma_2 \gamma_1 \cdot x)$ the \textit{parallel} edge to $e$.

In what follows, we will sometimes speak of edges as if they are sets of vertices, e.g, ``the distance between two edges" or ``these two edges are disjoint". The meaning of such statements will be the one resulting from identifying an edge $(x,y)$ with the set $\{x,y\}$.

\begin{figure}
\centering
\includegraphics[width=\textwidth]{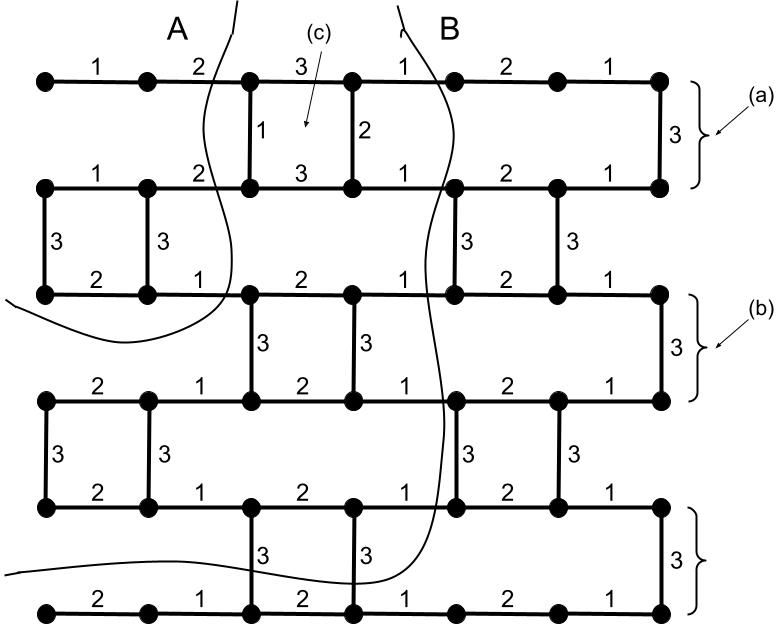}
\caption{\label{fig:general} A visualization of the setup and proof of Lemma \ref{lem:general}. The horizontal edges correspond to $\gamma_1$, and the vertical to $\gamma_2$. The edges pictured are those not given the color $4$. $A$ and $B$ are the regions indicated. The brackets show pairs in $P$. (c) shows the configuration of edges mentioned in condition 5 of the lemma statement. That configuration is used to swap parity for the element of $P$ labeled by (a), as shown. (b) labels an element of $P$ on which no parity swap is needed.}
\end{figure}

\begin{lemma}\label{lem:general}
Suppose we have a partial edge coloring $c:(G(a,S') \res U) \rightarrow \{1,2,3,4\}$ which satisfies the following properties. 
\begin{enumerate}
    \item The domain of $c$ consists of all $\gamma_2$-edges, and all $\gamma_1$-edges meeting $A \cup B$.
    \item $c$ gives all $\gamma_1$-edges the color $1$ or $2$, and all $\gamma_2$-edges the color $3$ or $4$.
    \item Parallel $\gamma_1$-edges are always given the same color.
    \item If $e,e'$ are two edges in a single element of $\F$ meeting $A$, they get the same color iff the distance between them is odd. Likewise for $B$.
    \item For every path from $A$ to $B$ consisting only of $\gamma_1$-edges, there is some edge $e$ on the path such that $e$ and its parallel edge, say $e'$, are both disjoint from $A \cup B$, and both $\gamma_2$-edges from $e$ to $e'$ have the color 3.  
\end{enumerate}
Then there is an (total) edge coloring $c':(G(a,S') \res U) \rightarrow \{1,2,3,4\}$ which agrees with $c$ on all edges meeting $A \cup B$.
\end{lemma}

\begin{proof}
Essentially a proof by picture; See Figure \ref{fig:general}. 

First, the move from $c$ to $c'$ will not change which edges get the color 4, so ignore those edges for the rest of the proof. We are left with a graph like the one in the figure. Fix a pair $\{f,f'\} \in P$, so that we need to fill in the colors of the edges in $f$ and $f'$ not meeting $A \cup B$. We would like to simply do this by continuing to alternate between the colors 1 and 2 on them, but we may run into parity issues.

By condition 4, this can only be the case along paths in $f$ or $f'$ between $A$ and $B$. By condition 5, though, whenever we have such a situation, we can ``borrow'' the color 3 to indeed do a parity swap if necessary. Item (a) in the figure shows an example where this is the case, and item (b) an example where it is not. 
\end{proof}

\subsubsection{The standard generators}

The remainder of Subsection \ref{subsec:2} will involve reducing our desired result to several concrete cases. All will be handled using similar overall strategies, and ultimately relying on Lemma \ref{lem:general}. The first of these will be when $S = \{\pm e_i \mid 1 \leq i \leq d\}$ is the usual generating set for $\Z^d$. Fix this particular $S$ for now. We will cover this case carefully to illustrate our general strategy.

We first set up some additional terminology. Let $U \subset U'$ be subsets of some fixed $\Z^d$-orbit $U''$, and let $c:(G(a,S) \res U') \rightarrow \{1,\ldots,2d\}$ an edge coloring. Let $x_0 \in U''$, $1 \leq i \leq d$, and $k,k'$ distinct colors. We say that $c$ \textit{follows the protocol $x_0$ for $e_i$ using $k$ and $k'$ on U} if for each edge $(x,e_i \cdot x)$ of $G(a,S) \res U'$ meeting $U$, $c$ gives this edge the color $k$ if $a_i$ is even and $k'$ otherwise, where $(a_1,\ldots,a_d) \in \Z^d$ is the unique element such that $(a_1,\ldots,a_d) \cdot x_0 = x$. Note for later that this definition makes sense for any $S$ containing $e_i$. Sometimes the ``on $U$'' will be omitted if $U$ is clear from context, which typically will be when $U = U'$. Likewise, the colors may be omitted if they are clear or if there is no need to name them.

Of course, given $x_1,\ldots,x_d \in U''$, it is always possible to $2d$-edge color a given subset of $U''$ following the protocol $x_i$ for $e_i$ using the colors $2i-1$ and $2i$ for each $i$. Also note that the choice of each $x_i$ only matters modulo the action of the subgroup $\langle 2 e_i \rangle$.

\begin{figure}
\centering
\includegraphics[width=0.4\textwidth]{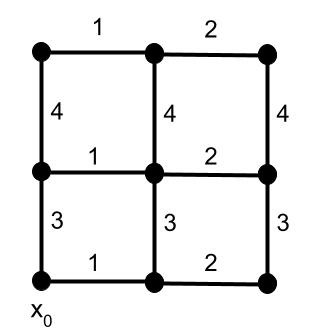}
\caption{\label{fig:protocol} A partial coloring for $d=2$ following the protocol $x_0$ for $e_1$ using the colors 1 and 2 and also following the protocol $x_0$ for $e_2$ using the colors 3 and 4. }
\end{figure}

It is clear that if $U$ and $V$ are two disjoint subsets of $U''$ and $c$ is a coloring of $G(a,S) \res (U \cup V)$ following the protocol $x_i$ for $e_i$ using the colors $2i-1$ and $2i$ for each $i$,  then $c$ can be extended to a coloring of all of $G(a,S) \res U''$ following those same protocols. We now need to see what to do when $U$ and $V$ follow different sequences of protocols. 

\begin{lemma}\label{lem:1transition}
Let $1 \leq i \leq d$. Suppose $U \subset U''$ and $c$ is an edge coloring of $G \res B(U,1)$ which for each $1 \leq j \leq d$ follows some protocol $x_j$ for $e_j$ using $2j-1$ and $2j$ on $A(U,0,1)$. Then we can extend $c$ to a $2d$-coloring of $G \res B(U,5)$ which still follows those protocols on $A(U,0,1)$, and which on $A(U,4,5)$ follows those protocols for each $j \neq i$, but follows protocol $e_i \cdot x_i$ for $e_i$ using $2i-1$ and $2i$.
\end{lemma}

\begin{proof}
First, color any edges in $B(U,5)$ meeting $A(U,4,5)$ or $A(U,0,1)$ as dictated our sequences of protocols. These two sequences agree about which protocols to use for $e_j$, $j \neq i$, so tentatively color all $e_j$-edges for $j \neq i$ in $A(U,0,5)$ as dictated by these protocols.

We need to fill in the colors for the remaining $e_i$ edges: those contained in $A(U,1,4)$. Since $d > 1$, fix some $j \neq i$. Observe that we can do this by applying Lemma \ref{lem:general} with $\gamma_1 = e_i$, $\gamma_2 = e_j$, $A = A(U,0,1)$, $B = A(U,4,5)$ and the colors 1,2,3, and 4 replaced with $2i-1,2i,2j-1$, and $2j$ respectively. All conditions other than 5 are immediately apparent. $P$ will pair up any two adjacent $e_i$ orbits whose connecting edges all get the color $2j-1$. That is, it will consist of pairs of the form $f,e_j \cdot f$ for which $f$ contains a point $(a_1,\ldots,a_d) \in x_0$ with $a_j$ even. Then, for condition 5, any edge $e$ such that both $e$ and its parallel edge are contained in $A(U,1,4)$ will work. 4 is large enough so that we are always guaranteed the existence of such an edge, so we are done. Note that the picture here will look like that in Figure \ref{fig:Z^d}.
\end{proof}

\begin{lemma}\label{lem:fulltransition}
Let $x_1,\ldots,x_d$, $y_1,\ldots,y_d \in U''$. Suppose $U \subset U''$ and $c$ is an edge coloring of $G(a,S) \res U$ which for each $1 \leq i \leq d$ follows protocol $x_i$ for $e_i$ using $2i - 1$ and $2i$. Then we can extend $c$ to a $2d$-coloring of $G \res B(U,4d+1)$ which still follows these protocols on $U$, but also for each $i$ follows protocol $x_i$ for $e_i$ using $2i - 1$ and $2i$ on $A(4d,4d+1)$.
\end{lemma}

\begin{proof}
This is an obvious induction using Lemma \ref{lem:1transition}.
\end{proof}

\begin{lemma}\label{lem:2sets}
Let $x_1,\ldots,x_d$, $y_1,\ldots,y_d \in U''$. Suppose $U,V \subset U''$, the $G(a,S)$-path distance between $U$ and $V$ is greater than $4d+1$, and $c$ is an edge coloring of $G \res (U \cup V)$ which on $U$, follows protocol $x_i$ for $e_i$ using $2i-1$ and $2i$ for each $i$, and likewise for $V$ and the $y_i$'s. Then we can extend $c$ to a $2d$-coloring of $U''$ which still follows those same protocols on $U$ and $V$ respectively.
\end{lemma}

\begin{proof}
First, extend $c$ to $B(U,4d+1)$ as in Lemma \ref{lem:fulltransition}. Then color the rest of the edges in $U''$ following protocol $y_i$ for $e_i$ using $2i-1$ and $2i$ for each $i$.
\end{proof}

Now, thanks to the fact $\asi_B(G(a,S)) = 1$, Lemma \ref{lem:2sets} is enough:

\begin{lemma}\label{lem:std}
With $S$ the standard generating set, $\chi_B(G(a,S)) = |S| = 2d$.
\end{lemma}

\begin{proof}
Let $N = 8(d+1)$. Let $X = U_1 \sqcup V$ be a Borel partition witnessing $\asi_B(G(a,S)) = 1$ for this value of $N$. Let $V_1 = X \setminus B(U_1,4d+1)$. Arguing as in the proof of Lemma \ref{lem:rainbow}, we can conclude that every injective $G(a,S)$-ray passes through $V_1$. Thus, by K\"{o}nig's lemma, the connected components of $G(a,S) \res (X \setminus V_1)$ are still all finite. $G(a,S) \res V_1$ also has only finite connected components since $G(a,s) \res V$ did. Likewise for $U_1$ and its complement. See Figure \ref{fig:toast} to see how these sets might be arranged.

\begin{figure}
\centering
\includegraphics[width=0.8\textwidth]{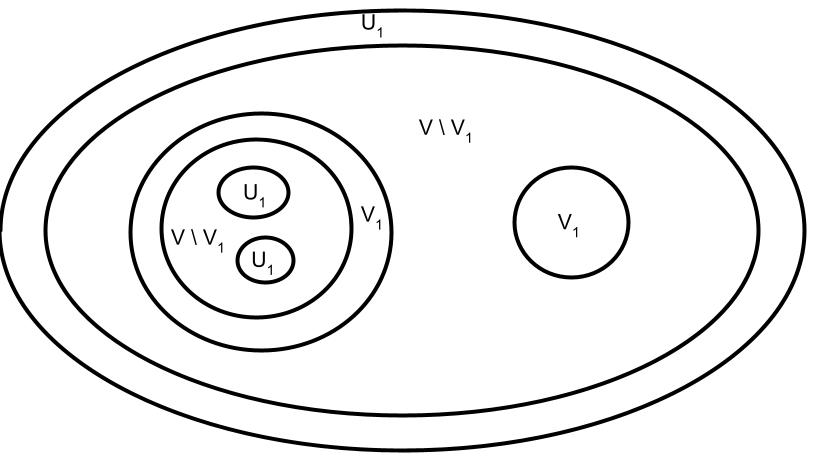}
\caption{\label{fig:toast} A typical arrangement of the sets in the proof of Lemma \ref{lem:std} for $d=2$.}
\end{figure}

Thus, as in the proof of Lemma \ref{lem:Z}, we may Borel $2d$-edge color $G(a,S) \res U_1$ so that a consistent sequence of protocols having the form from Lemmas \ref{lem:fulltransition} and \ref{lem:2sets} is followed within each $G(a,S) \res (X \setminus V_1)$ component, and Borel $2d$-edge color $G(a,S) \res V_1$ so that a consistent such sequence of protocols is followed within each $G(a,S) \res (X \setminus U_1)$-component.

Now let $C$ be a $G(a,S) \res (V \setminus V_1)$-component. Let $V_2$ be the union of the $V_1$-components meeting $B(C,1)$ (that is, with edges to $C$),and $U_2$ the union of the $U_1$-components meeting $B(C,1)$. $V_2$ is contained in a single $(X \setminus U_1)$-component, and hence is colored using a consistent sequence of protocols having the form from Lemma \ref{lem:2sets} , and likewise for $U_2$. Thus, by the lemma (note that the path distance between $U_2$ and $V_2$ is more than $4d+1$), we can find a $2d$-coloring of $G \res (C \cup U_2 \cup V_2)$ which still follows the original protocols set for $U_2$ and $V_2$ on those sets. Since $C$ is always finite, we can in a Borel fashion do this simultaneously for every such $C$, and by construction this does not cause any color conflicts.
\end{proof}

\subsubsection{Protocols and Transitions}\label{subsubsec:pro}

Since the remaining cases we cover will be handled very similarly to this one, we will not go over these details for each one. Instead we will content ourselves in each one with establishing some appropriate notions of ``protocol'', and using Lemma \ref{lem:general} to at least sketch some analogue of Lemma \ref{lem:fulltransition} saying different sequences of such protocols can be transitioned between in a bounded neighborhood. 

What we called ``protocol'' in the previous arguments, let us now call ``\textit{standard protocol}'', to differentiate it from the variants to come. Many such variants will be defined, so they are collected in Table \ref{tab:protocols} to help the reader keep track of them. The general form of these protocols will be the same as for the standard protocol: They will always specify some periodic way of 2-coloring the $\gamma$-edges in some orbit of some group relative to some basepoint in that orbit for $\gamma$ some specified generator of the group. Some protocols will only make sense for certain groups and/or certain values of $\gamma$.

The syntax for specifying $\gamma$ and the colors will be the same as for the standard protocol. As with the standard protocol, when we say this protocol is \textit{followed} on some set of vertices $U$, it will mean that all $\gamma$-edges meeting $U$ are colored in the ``way'' specified by the protocol.

Let us now introduce some general terminology for the ``transitions'' between protocols mentioned three paragraphs ago. Let $U''$ be an orbit of some group action $a$, $\gamma_1,\ldots,\gamma_n$ generators, $\Pi_1,\ldots,\Pi_n, \Pi_1',\ldots,\Pi_n'$ protocols (including basepoints), where for each $i$, $\gamma_i$ is the type of group element which $\Pi_i$ and $\Pi_i'$ describe colorings for. We say \textit{we can transition from $\Pi_1$ for $\gamma_1$,$\ldots$, $\Pi_n$ for $\gamma_n$ to $\Pi_1'$ for $\gamma_1$,$\ldots$, $\Pi_n'$ for $\gamma_n$} if roughly speaking, given a $2n$-coloring $c$ of $G(a,\{\gamma_1,\ldots,\gamma_n\})$ restricted to some subset $U$ of $U''$ which for each $i$ follows $\Pi_i$ for $\gamma_i$ using, say, the colors $2i-1$ and $2i$, we can extend $c$ to a $2n$-coloring of the restriction to some bounded neighborhood of $U$ which on the ``boundary'' of that neighborhood follows $\Pi_i'$ for $\gamma_i$ using the colors $2i-1$ and $2i$ for each $i$. A more precise statement would look something like the statements of Lemmas \ref{lem:1transition} and \ref{lem:fulltransition}.

The most basic claims of the form described in the previous paragraph will be proven using a single application of Lemma \ref{lem:general}, with some distinct $\gamma_i$, $\gamma_j$ from the previous paragraph playing the part of $\gamma_1$ and $\gamma_2$ from the lemma. More difficult ones will be built up in steps from several of these ``basic'' ones, using additional sequences of protocols as intermediaries.

\begin{table}[]
    \centering
    \begin{tabular}{| c | c | c | c |}
        \hline 
        Protocol & Context & Generators applied to & Reference \\
        \hline
        \hline
        Standard & $\Z^d$, $d \geq 2$ & $e_i$, $1 \leq i \leq d$ & Figure \ref{fig:protocol} \\
        \hline
        Block & $\Z^2$ & $(a,0)$ or $(0,a)$, $a > 0$ & Section \ref{subsubsec:block} \\
        \hline
        Diagonal & $\Z^2$ & $(b_1,b_2)$, $b_1,b_2 > 0$ & Beginning of Section \ref{subsubsec:3gen1} \\
        \hline
        Alternating & $\Z^2$ & $(0,1)$ & Figure \ref{fig:diag} \\
        \hline
    \end{tabular}
    \caption{The different protocols used in Section \ref{subsec:2}. Column two gives the ambient group which each protocol is defined for. Column three lists those group elements whose edges the protocol sets colors for.}
    \label{tab:protocols}
\end{table}

let us use the already covered case of $\Z^d$, with the standard generators as an example to showcase our terminology. After defining the standard protocol and fixing a $\Z^d$-orbit $U''$, we could replace Lemma \ref{lem:1transition} with the following statement, proved the same way with a single application of Lemma \ref{lem:general}.

\begin{lemma}
Let $i \neq j$, $x_i,x_j \in U''$. We can transition from standard protocol $x_i$ for $e_i$ and standard protocol $x_j$ for $e_j$ to standard protocol $e_i \cdot x_i$ for $e_i$ and standard protocol $x_j$ for $e_j$.
\end{lemma}

Inductively applying this would give the following restatement of Lemma \ref{lem:fulltransition}:

\begin{lemma}
Let $x_1,\ldots,x_d,y_1,\ldots,y_d \in U''$. We can transition from standard protocol $x_i$ for $e_i$ for each $i$ to standard protocol $y_i$ for $e_i$ for each $i$.
\end{lemma}

\subsubsection{Multiples of the standard generators}\label{subsubsec:block}

One easy case we will need to address is $d = 2$, $$S = \{\pm (a_1,0),\ldots,\pm (a_n,0), \pm(0,b_1),\ldots,\pm(0,b_m)\}$$ for $n,m > 1$ and the $a_i$'s and $b_i$'s arbitrary positive integers (possibly with repeats). 

Let $U''$ be a $\Z^2$-orbit as before, $x_0 \in U''$, $a > 0$. The \textit{block protocol} $x_0$ for $(a,0)$ using, say, the colors 1 and 2, will dictate that if $x = (g_1,g_2) \cdot x_0$, then the edge $(x,(a,0) \cdot x)$ gets the color $1$ if $g_1 \in [0,a)$ (mod $2a$), and the color 2 otherwise. The block protocol $x_0$ for $(0,a)$ is defined similarly. Note that the standard protocols are a special case of these.

\begin{lemma}\label{lem:blocktrans}
Let $a,b > 0$, $x_0,y_0 \in U''$. We can transition from block protocol $x_0$ for $(a,0)$ and block protocol $y_0$ for $(0,b)$ to block protocol $(1,0) \cdot x_0$ for $(a,0)$ and block protocol $y_0$ for $(0,b)$.
\end{lemma}
\begin{proof}
let us say our protocols use the colors 1 and 2 for $(a,0)$ and 3 and 4 for $(0,b)$. As in the proof of Lemma \ref{lem:1transition}, we apply Lemma \ref{lem:general} with $\gamma_1 = (a,0)$, $\gamma_2 = (0,b)$, and $P$ pairing up adjacent $(a,0)$-orbits whose connecting edges all get the color 3. This time, though, we only need to change parity along those $(a,0)$-orbits containing a point of the form $(g_1,g_2) \cdot x$ with $g_1$ a multiple of $a$. An orbit $f$ satisfies this condition if and only if its adjacent orbits do, so this is fine. Note that condition 3 from the lemma holds since in a block protocol for $(a,0)$, the color of the edge $((g_1,g_2) \cdot x_0, (g_1+a,g_2) \cdot x_0)$ only depends on the first coordinate $g_1$.
\end{proof}

Of course, the analogous statement with the coordinates swapped hold as well. Therefore, since $n,m > 0$, an inductive argument shows that we can transition between any pair of sequences of block protocols for our generators in $S$, and so we conclude:

\begin{lemma}\label{lem:mult}
With $d=2$ and $S$ as above, $\chi_B(G(a,S)) = |S|$.
\end{lemma}

Of course, there is nothing special here about $d=2$, but we will not need the higher dimensional versions of this case.

\subsubsection{$\Z^2$ with three generators, subcase 1}\label{subsubsec:3gen1}

We now consider cases of the form $d = 2$, $S = \{\pm(n,0),\pm(0,n),\pm(b_1,b_2)\}$, where $n,b_1$, and $b_2$ are positive integers. Of course, we may assume $n,b_1$, and $b_2$ do not share any common factors. Fix such an $S$ for the time being.

Let $U'', x_0$ as before. The \textit{diagonal protocol} $x_0$ for $(b_1,b_2)$, say using the colors 5 and 6, will dictate that edges of the form $((a_1,a_2) \cdot x_0 , (a_1+b_1,a_2+b_2) \cdot x_0)$ get the color 5 if $a_1 \in [0,b_1)$ (mod $2b_1$) and 6 otherwise. We will always use such protocols for $(b_1,b_2)$.

First consider the subcase $n = 1$. If $b_1$ and $b_2$ are not both odd, assume WLOG $b_2$ is even. Outside of transitional steps, we will use standard protocols for the generators $(1,0)$ and $(0,1)$. We have already seen that, with the standard protocols, these generators can use each other to change their basepoints. Since the diagonal protocol only depends on the first coordinate of its basepoint, it therefore suffices to show:

\begin{lemma}\label{lem:3gentrans1}
Let $x_0 \in U''$. We can transition from the standard protocol $x_0$ for $(1,0)$ and $(0,1)$ and the diagonal protocol $x_0$ for $(b_1,b_2)$ to the standard protocol $x_0$ for $(1,0)$ and $(0,1)$ and the diagonal protocol $(1,0) \cdot x_0$ for $(b_1,b_2)$.
\end{lemma}

\begin{proof}
This will be split into two further subcases, depending on the parity of $b_1$ and $b_2$. For the first, suppose $b_2$ is even.

let us say our protocols for $(0,1)$ use the colors 3 and 4. We want to apply Lemma \ref{lem:general} with $\gamma_1 = (b_1,b_2)$ and $\gamma_2 = (0,1)$. Since $b_2$ is even, if an edge between two adjacent $(b_1,b_2)$-orbits is given the color 3 by the standard protocol $x_0$ for $(0,1)$, then all such edges are. Thus we can apply the lemma exactly as we did in the proof of Lemma \ref{lem:blocktrans}.

For the second, by our assumption from before the statement of the lemma, both $b_1$ and $b_2$ are odd. The \textit{alternating protocol} $x_0$ for $(0,1)$, say using the colors 3 and 4, will dictate that edges of the form $((a_1,a_2) \cdot x_0, (a_1,a_2+1) \cdot x_0)$ get the color dictated by the standard protocol $x_0$ using 3 and 4 if $a_1 \in \{0,1\}$ (mod 4), and the opposite of that color otherwise. See Figure \ref{fig:diag}. 

Our desired transition will be accomplished in three ``steps'' in the sense of Section \ref{subsubsec:pro}. First we switch to an alternating protocol for $(0,1)$, then shift the basepoint for $(b_1,b_2)$'s diagonal protocol as desired, then switch back to the standard protocol for $(0,1)$.

For the first and third of these steps (noting our ``we can transition...'' relation is symmetric) we show that we can transition from standard protocol $x_0$ for $(1,0)$ and $(0,1)$ to standard protocol $x_0$ for $(1,0)$ and alternating protocol $x_0$ for $(0,1)$. This is exactly like the transition to standard protocol $(0,1) \cdot x_0$ for $(0,1)$, but in our application of Lemma \ref{lem:general}, instead of changing parity along each pair of $(0,1)$-orbits in $P$, we only change parity along every other pair.

For the second, we show that we can transition from alternating protocol $x_0$ for $(0,1)$ and diagonal protocol $x_0$ for $(b_1,b_2)$ to alternating protocol $x_0$ for $(0,1)$ and diagonal protocol $(1,0) \cdot x_0$ for $(b_1,b_2)$. Figure \ref{fig:diag} shows the configuration given by our starting set of protocols. let us say we are using the same colors as in the figure. Consider removing the $(0,1)$-edges colored 4 from this picture. Since $b_1$ and $b_2$ are odd, the resulting picture looks exactly like that in Figure \ref{fig:general}, but with the colors 5 and 6 in place of the colors 1 and 2. Therefore condition 5 from Lemma \ref{lem:general} is satisfied with $\gamma_1 = (b_1,b_2)$, $\gamma_2 = (0,1)$. Condition 3 is satisfied as in the previous subcase. Thus we can swap parity along the appropriate pairs of adjacent $(b_1,b_2)$-orbits. (As in the proof of Lemma \ref{lem:blocktrans}, this will be those whose orbits contain a point of the form $(a_1,a_2) \cdot x_0$ with $a_1$ a multiple of $b_1$.)
\end{proof}

\begin{figure}
\centering
\includegraphics[width=\textwidth]{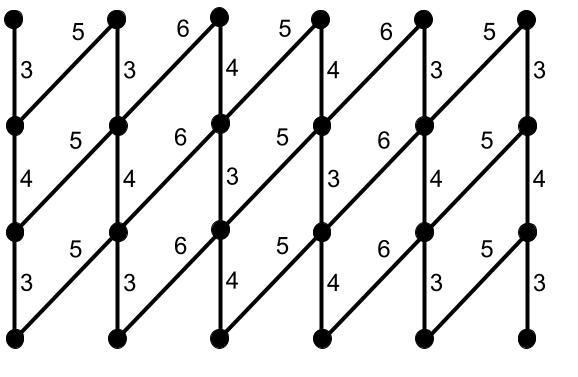}
\caption{\label{fig:diag} A picture of a coloring using an alternating protocol for $(0,1)$ and a diagonal one for $(b_1,b_2)$, as in the second case in the proof of Lemma \ref{lem:3gentrans1}. Here $b_1=b_2=1$.}
\end{figure}

\subsubsection{$\Z^2$ with three generators, subcase 2}\label{subsubsec:3gen2}

Next consider the subcase $n > 1$. Then, since we assumed there were no common divisors, WLOG $n$ does not divide $b_2$. We will use block protocols for the $(n,0)$ and $(0,n)$ edges.

Once again, since we already have Lemma \ref{lem:blocktrans}, we need only to show:

\begin{lemma}\label{lem:3gentrans2}
Let $x_0 \in U''$. We can transition from the block protocol $x_0$ for $(0,n)$ and the diagonal protocol $x_0$ for $(b_1,b_2)$ to the block protocol $x_0$ for $(0,n)$ and the diagonal protocol $(1,0) \cdot x_0$ for $(b_1,b_2)$.
\end{lemma}

\begin{figure}
\centering
\includegraphics[width=\textwidth]{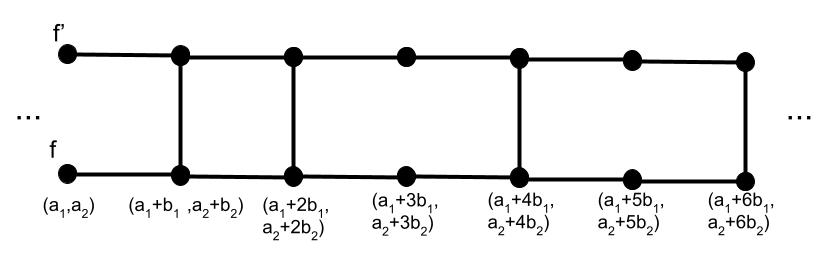}
\caption{\label{fig:uneven} A visualization of the $n > 1$ case. The horizontal edges are the $(b_1,b_2)$-edges, while the vertical ones are the $(0,n)$-edges colored 3. The label $(g_1,g_2)$ for a point is shorthand for $(g_1,g_2) \cdot x_0$. Whether or not a vertical edges touching some $(g_1,g_2)$ is included depends on the value of $g_2$ mod $2n$.}
\end{figure}

\begin{proof}
We wish to apply Lemma \ref{lem:general} with $\gamma_1 = (b_1,b_2)$ and $\gamma_2 = (0,n)$, let us say using the colors 5 and 6 for $\gamma_1$ and 3 and 4 for $\gamma_2$. Condition 3 is satisfied as in previous arguments. To apply the lemma in the way we want, it would suffice to show the following about a coloring using our starting protocols: There is an $N$ (depending only on $n,b_1$, and $b_2$) such that for every $(b_1,b_2)$-orbit $f$ with adjacent orbit $f' = (0,n) \cdot f$, for every path in $f$ of length $N$, there is some edge $e$ in the path satisfying condition 5 of the lemma. That is, such that both $(0,n)$-edges from $e$ to $f'$ get the color 3.

Let $(a_1,a_2) \cdot x_0$ be an arbitrary point in $f$, so that the points on $f$ are those of the form $(a_1 + kb_1,a_2+kb_2) \cdot x_0$ for $k \in \Z$. By definition of the block protocol, the edge from such a point to $f'$ will get the color 3 if and only if $a_2 + kb_2 \in [0,n)$ (mod $2n$). The picture so far is summed up in Figure \ref{fig:uneven}.

Our problem thus reduces to the following number theoretic one. Given that $n \not|\ b_2$, find $N$ large enough so that for any $a_2 \in \Z$, there is some $k \in [0,N]$ such that $a_2+kb_2$ and $a_2 + (k+1)b_2$ are both in the range $[0,n)$ (mod $2n$). Since everything is modulo $2n$, it just suffices to show that we can find some such $k \in \N$, and then we can take $N = 2n$. Assume that $b_2 \in (0,n)$ (mod $2n$), as the case $b_2 \in (n,2n)$ (mod $2n$) will follow from a similar argument.

Now let $q = \gcd(b_2,2n) < n$. Then there will always be a $k$ such that $a_2 + kb_2 \in [0,q)$ (mod $2n$). Also, since we have assumed $b_2 \in (0,n)$ (mod $2n$), we must actually have $b_2 \in [q,n-q]$ (mod $2n$). Then $a_2+(k+1)b_2$ is in the range $[0,n)$ (mod $2n$) as desired.
\end{proof}

Let us record the result of Subsubsections \ref{subsubsec:3gen1} and \ref{subsubsec:3gen2} together:

\begin{lemma}\label{lem:3gens}
With $d = 2$, $S = \{\pm(n,0),\pm(0,n),\pm(b_1,b_2)\}$ for some $b_1,b_2,n > 0$, we have $\chi_B'(G(a,S)) = |S| = 6$.
\end{lemma}

\subsubsection{Combining the cases}\label{subsubsec:2comb}

Finally, we are ready to finish the proof for general $d$ and $S$.

\begin{proof}
Let $S = \{\pm \gamma_1,\ldots,\pm \gamma_n\}$, and let $S' = \{\gamma_1,\ldots,\gamma_n\}$. We define an equivalence relation $\sim$ on $S'$ which sets two generators as equivalent if one is a scalar multiple of the other. Equivalently, if the subgroup they generate is cyclic. Let $C_1,\ldots,C_m$ list the classes in $S'/\sim$. Note that we must have $m \geq d > 1$.

Because of this, we can partition the set of $C_i$'s into sets of size 2 and 3. It now suffices to show that given an element of this partition $\{C_{i_j} \mid 1 \leq j \leq r\}$, where $r \in \{2,3\}$, we can Borel color the edges in $G(a,\pm \bigcup_j C_{i_j})$ using $|\pm \bigcup_j C_{i_j}|$ colors.

First consider the case $r = 2$. For notational convenience, take $i_j = j$ for $j=1,2$. Let $\Gamma = \langle C_1,C_2 \rangle \cong \Z^2$. Then $(\Gamma,\pm (C_1 \cup C_2))$ is isomorphic as a marked group to $\Z^2$ with a set of generators of the form covered in Lemma \ref{lem:mult}, so that lemma says this case is done.

Next consider the case $r=3$, and again take $i_j = j$ for $j=1,2,3$. We would like to reduce to the case where $C_1,C_2$, and $C_3$ are all singletons. Suppose, for instance that $C_1$ includes at least two elements, and let $\delta_1$ be one. Pick some $\delta_2$ from $C_2$. Then we can $4$-edge color $G(a,\{\pm \delta_1,\pm \delta_2\})$ by the previous paragraph. We are left with $C_1 - \{\delta_1\}, C_2 - \{\delta_2\}$, and $C_3$. Note that $C_1 - \{\delta_1\}$ is still nonempty. If $|C_2| = 1$, then we have reduced to the previous paragraph. If not, and one of our three sets still contains more than one element, repeat this.

Thus we have our reduction. Let $C_i = \{\delta_i\}$ for $i=1,2,3$. Now by definition of $\sim$, $\Gamma := \langle \delta_1,\delta_2,\delta_3 \rangle$ is isomorphic to either $\Z^2$ or $\Z^3$. 

In the latter case, $(\Gamma,\{\pm \delta_1,\pm \delta_2, \pm \delta_3\})$ is isomorphic as a marked group to $\Z^3$ with its standard generators, so by Lemma \ref{lem:std} we are done.

In the former case, we can write $n \delta_3 = b_1 \delta_1 + b_2 \delta_2$ for some $n,b_1,b_2 \neq 0$. Then $(\Gamma,\{\pm \delta_1,\pm \delta_2, \pm \delta_3\})$ is isomorphic as a marked group to $\Z^2$ with the generating set from Lemma \ref{lem:3gens}. (Send $\delta_1$ to $(n,0)$, $\delta_2$ to $(0,n)$, and $\delta_3$ to $(b_1,b_2)$.) We can also assume $n,b_1,b_2 > 0$ by replacing the $\delta$'s with their inverses if necessary. Thus by the lemma, we are done.
\end{proof}

\subsection{Free rank 1}\label{subsec:1}

In this subsection we complete the proof of Statement 2 of Theorem \ref{th:Z^d}. For the discrete part of the statement, since we can use Lemma \ref{lem:quotient} with $\Delta = \Delta'$, it suffices to prove $\chi'(G(a,S)) = |S|$ when $\Gamma = \Z$, but where we now allow multiplicity for generators. This is obviously the case, though.

We thus turn to the Borel part of the statement, We need to show that if $\Delta$ has even order, then $\chi_B'(G(a,S)) = |S|$. As in the proof of Statement 1, we can start by finding an index 2 subgroup $\Delta' \leq \Delta$, and then use Lemma \ref{lem:quotient} to reduce to the case where $\Delta = \Z/2$, but where we now allow multiplicity for generators in $S$. For the rest of the subsection, assume $(\Gamma,S)$ has this form. We will sometimes write ``$\Z$'' to refer to the subgroup $\{0\} \times \Z \leq \Gamma$.

We start with the special case where $(1,0) \in S$. (Recall we are writing $\Gamma = \Z/2 \times \Z$. For example this element has order 2.)

\begin{lemma}\label{lem:(1,0)}
If $(1,0) \in S$, $\chi_B'(\Gamma,S) = |S|$.
\end{lemma}

\begin{proof}
This follows immediately from Corollary \ref{cor:d} with $\Delta = \Z/2$.
\end{proof}

Thus, from now on assume $(1,0) \not\in S$. Then we can write $S = S' \sqcup -S'$, where $S'$ consists of all pairs $(\epsilon,n) \in S$ for which $n > 0$. The elements $(\epsilon,n) \in S'$ will be organized into 4 categories corresponding to the value of $\epsilon$ and the parity of $n$. We will consider these categories roughly one at a time, for each establishing suitable notions of protocol as in the previous subsection, and demonstrating how to transition between protocols. See Table \ref{tab:1protocols} for a list of protocols to be used in this subsection.

These transitions will almost always be accomplished by an argument using Lemma \ref{lem:general}. That lemma was stated in the context of $\Z^d$, but in this context its terminology still makes sense and its conclusion still holds.

\begin{table}[]
    \centering
    \begin{tabular}{| c | c | c | c |}
        \hline 
        Protocol & Context & Generators applied to & Reference \\
        \hline
        \hline
        Two sided & $\Z/2 \times \Z$ & $(1,m)$, $m > 0$ & Figure \ref{fig:altpar} \\
        \hline
        Odd & $\Z$ & $n > 0$ odd & Figure \ref{fig:altpar} \\
        \hline
        Parallel & $\Z/2 \times \Z$ & $(0,n)$, $n >0$ odd & Figure \ref{fig:altpar}(b)\\
        \hline
        Alternating & $\Z/2 \times \Z$ & $(0,n)$, $n > 0$ odd & Figure \ref{fig:altpar}(a) \\
        \hline
        Block & $\Z$ & $n > 0$ odd & Figure \ref{fig:uneven2}(b) \\
        \hline
        Offset & $\Z/2 \times \Z$ & $(0,n)$, $n > 0$ odd & Figure \ref{fig:uneven2}(b)\\
        \hline
        One code & $\Z$ & $b > 0$ even &  Before Lemma \ref{lem:eventrans}\\
        \hline
        Short code & $\Z$ & $b > 0$ even & After the proof of Lemma \ref{lem:eventrans} \\
        \hline
        Two code & $\Z/2 \times \Z$ & $(0,b)$, $b > 0$ even & Figure \ref{fig:even}\\
        \hline
        
    \end{tabular}
    \caption{Like Table \ref{tab:protocols}, but for free rank 1. Note that the offset, one code, short code, and two code protocols require additional data to be fully specified.}
    \label{tab:1protocols}
\end{table}

\subsubsection{Generators from $2\Z + (0,1)$ and $\Z + (1,0)$}\label{subsubsec:rank1base}

We start by examining how to color edges of the form $(0,n)$ for $n$ an odd positive integer and $(1,m)$ for $m$ any positive integer. Fix such $n$ and $m$. Let $U''$ be a $\Gamma$-orbit, and $x_0 \in U''$. 

The \textit{two sided protocol} $x_0$ for $(1,m)$, say using the colors 3 and 4, will dictate that edges of the form $(x,(1,m) \cdot x)$ get the color 3 if $x$ and $x_0$ are in the same $\Z$-orbit and 4 if they are not. See both parts of Figure \ref{fig:altpar} for examples.

For the $(0,n)$ edges, we start with one $\Z$-orbit: If $y_0$ is in an orbit of some free $\Z$-action, a coloring of $n$-edges in that orbit will be said to follow the \textit{odd protocol} $y_0$ for $n$, say using the colors 1 and 2, if the edge $(a \cdot y_0, (a+n) \cdot y_0)$ gets the color 1 if $a$ is even and 2 otherwise. Note that this does define a coloring since $n$ is odd.

Now in our original context, the \textit{parallel protocol} $x_0$ for $(0,n)$ will dictate that the odd protocols $x_0$ and $(1,0) \cdot x_0$ are followed on the $\Z$-orbits $\Z \cdot x_0$ and $\Z \cdot (1,0) \cdot x_0$ respectively. The \textit{alternating protocol} $x_0$ for $(0,n)$ is defined similarly, except that $(1,1) \cdot x_0$ is used as the basepoint for the odd protocol on the second $\Z$-orbit. See Figure \ref{fig:altpar} for examples.

\begin{figure}
\centering
\includegraphics[width=0.8\textwidth]{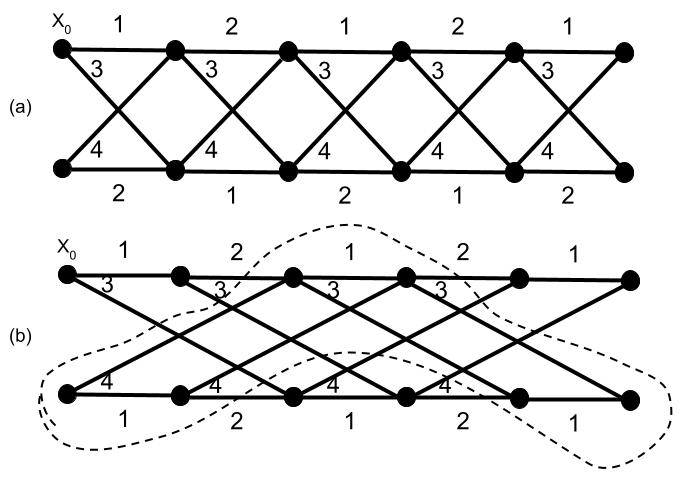}
\caption{\label{fig:altpar} (a): A drawing of a coloring following protocol $x_0$ in the case $n=m=1$. Note that the $(0,1)$-edges use the alternating protocol. (b): Same, but now with $m=2$. Note that now the $(0,1)$-edges use the parallel protocol. The dotted line sections off a pair of adjacent $(1,m)$ orbits. This makes it clear that they follow the necessary hypotheses to apply Lemma \ref{lem:general}.}
\end{figure}

Note that for two sided protocols, any choice of basepoint will give the same protocol as either $x_0$ or $(1,0) \cdot x_0$. Similarly, for parallel and alternating protocols, any choice of basepoint will give the same protocol as either $x_0$ or $(0,1) \cdot x_0$. 

\begin{lemma}\label{lem:1dbasictrans}
Let $x_0,y_0,x_1,y_1 \in U''$. Then
\begin{enumerate}
    \item If $m$ is even, then we can transition from the two sided protocol $x_0$ for $(1,m)$ and the parallel protocol $y_0$ for $(0,n)$ to the two sided protocol $x_1$ for $(1,m)$ and the parallel protocol $y_1$ for $(0,n)$.
    \item Likewise if $m$ is odd, but with alternating protocols in place of the parallel ones for $(0,n)$.
\end{enumerate}
\end{lemma}

\begin{proof}
Note that Figures \ref{fig:altpar} (b) and (a) provide a picture for parts 1 and 2 of the lemma respectively. Let us use the color assignments from those pictures.

By the comment before the statement of the lemma, it suffices to show these statements in the following two instances: When $x_0 = x_1$ and $y_1 = (0,1) \cdot y_0$, and when $y_0 = y_1$ and $x_1 = (1,0) \cdot x_0$.

For the first instance, we want to apply Lemma \ref{lem:general} with $\gamma_1 = (0,n)$ and $\gamma_2 = (1,m)$. $P$ will group together $(0,n)$-orbits $f$ and $(1,m) \cdot f$ if $f$ lies in the same $\Z$-orbit as $x_0$. This way, by definition of the two sided protocol, the $(1,m)$-edges from an edge $e$ in $f$ to its parallel edge always get the color 3. This gives condition 5 from the lemma. Condition 3 is ensured by our choice of parallel or alternating protocols for $(0,n)$ according to the parity of $m$. Thus we can use the lemma to swap parity along every $(0,n)$-orbit, which lets us transition to the protocol $(0,1) \cdot y_0$.

For the second instance, we want to apply Lemma \ref{lem:general} with $\gamma_1 = (1,m)$ and $\gamma_2 = (0,n)$. Observe that it is possible to define the partition $P$ so that if $\{f,f'\} \in P$ (so $f$ and $f'$ are $(1,m)$-orbits), the edges between them all get color 1. Part (b) of Figure \ref{fig:altpar} gives an example of how to do this. The idea is that we have chosen the protocol for our $(0,n)$-edges according to the parity of $m$ so that our coloring of them descends to one of the quotient of our graph by the subgroup $\langle (1,m) \rangle$, so we can group together orbits exactly if they are related in this quotient by an edge of color 1. This takes care of condition 5 from the lemma, and condition 3 comes from our use of two sided protocols.
\end{proof}

\subsubsection{Generators from $\Z + (1,0)$ with different parity}

As an example to motivate this subsubsection, suppose $S'$ consists of one element of the form $(0,n)$, for $n$ odd, and several of the form $(1,m)$ for $m$ even. It follows from Lemma \ref{lem:1dbasictrans} that if we use parallel protocols for $(0,n)$ and two sided protocols for the $(1,m)$'s, we can transition between different sequences of basepoints as needed, and therefore get a Borel $|S|$-coloring. 

Suppose, though, that we throw in a generator of the form $(1,b)$ for $b$ odd. To use Lemma \ref{lem:1dbasictrans} for this generator would require an alternating protocol for $(0,n)$, but we have already committed to using parallel protocols for it. Thus we need the following:

\begin{lemma}\label{lem:diffparity}
Let $m$,$n$ as in the previous subsubsection, with $m$ even. Let $b > 0$ odd. Let $x_0 \in U''$. We can transition from the parallel protocol $x_0$ for $(0,n)$ and the two sided protocols $x_0$ for $(1,m)$ and $(1,b)$ to the parallel protocol $x_0$ for $(0,n)$, the two sided protocol $x_0$ for $(1,m)$, and the two sided protocol $(1,0) \cdot x_0$ for $(1,b)$.
\end{lemma}

Throughout, let us use the colors 1 through 4 for $(0,n)$ and $(1,m)$ as in the previous subsubsection, and the colors 5 and 6 for $(1,b)$. Our proof split into cases depending on whether $n \mid b$. Suppose first that it does. Then, we can actually make do without the $(1,m)$-edge:

\begin{lemma}
If $n \mid b$, we can transition from the parallel protocol $x_0$ for $(0,n)$ and the two sided protocol $x_0$ for $(1,b)$ to the parallel protocol $x_0$ for $(0,n)$ and the two sided protocol $(1,0) \cdot x_0$ for $(1,b)$.
\end{lemma}

\begin{proof}
Essentially a proof by picture. By working separately on each $\Z/2 \times n\Z$-orbit, we can assume $n = 1$. Then, Figure \ref{fig:n=1} shows how to use the $(0,1)$-edges to transition between our two protocols for the $(1,b)$-edges in the case $b=3$. It should be clear how this pattern can be continued to other $b$.
\end{proof}

\begin{figure}
\centering
\includegraphics[width=\textwidth]{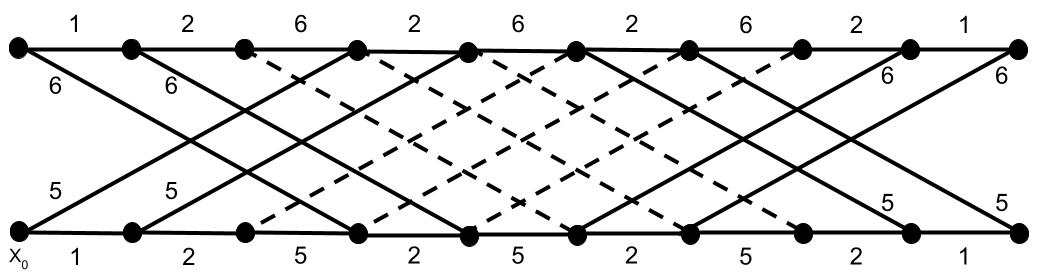}
\caption{\label{fig:n=1} A demonstration of how to change protocol for the $(1,b)$ edges when $n=1$ and $b=3$. The $(0,1)$-edges use parallel protocol $x_0$. The dashed diagonal lines all get the color 1. Observe that, on the left side of the figure, the $(1,b)$-edges use the two sided protocol $x_0$, while on the right they use the two sided protocol $(1,0) \cdot x_0$, as desired.}
\end{figure}

Next suppose that $n$ does not divide $b$. Let $q = \gcd(n,b) = \gcd(n,2b) < n$. The following will be important for a soon-to-come number theoretic argument.

\begin{lemma}\label{lem:farfromn}
$b+m$ and $b-m$ do not both lie in the interval $(n-q,n+q)$ (mod $2n$).
\end{lemma}

\begin{proof}
Suppose not. Then, since $m+b$ and $m-b$ differ by $2b$, a multiple of $q$, they must either be equal mod $2n$, or differ by exactly $q$ mod $2n$.

In the former case, we get $2b = 0$ (mod $2n$), which contradicts $n$ not dividing $b$.

In the latter case, we get $2b = \pm q$ (mod $2n$), but note that $q$ must be odd, so this is also a contradiction.
\end{proof}

Assume then that $b-m$ does not lie in this interval mod $2n$. The argument in the case where $b+m$ is not in this interval will be similar. Assume further that $b-m \in [0,n-q]$ (mod $2n$), as the argument in the case where it is in $[n+q,2n]$ (mod $2n$) will again be similar. With these assumptions, let us prove Lemma \ref{lem:diffparity}:

\begin{figure}
\centering
\includegraphics[width=1\textwidth]{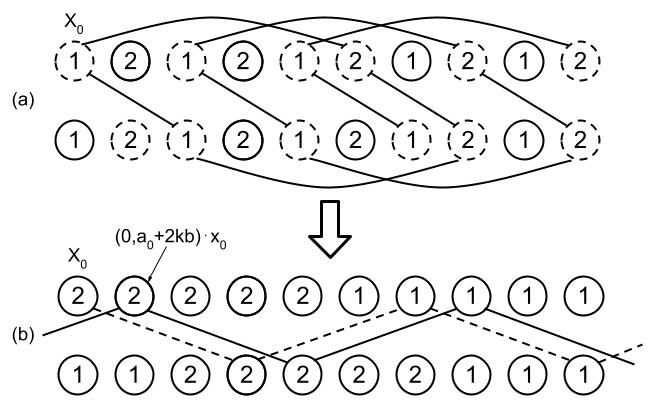}
\caption{\label{fig:uneven2} A demonstration of how to transition between protocols for the $(1,b)$ edges when $n=5$,$m=2$, and $b=3$. For visual clarity, the label of a vertex $x$ refers to the color of the edge $(x,(0,n) \cdot x)$. (a): Step 1 of the procedure. The dashed vertices are those included in one of the $(0,n)$-orbits along which parity is changed. The edges drawn are the $(0,n)$ and $(1,m)$-edges involved in making that change for these orbits. (b): The effects of step 1, along with a pair of $(1,b)$-orbits of the form $f,(0,n) \cdot f$ from step 2. The solid and dashed lines are the edges from $f$ and $(0,n) \cdot f$ respectively. A vertex $(0,a_0+2kb) \cdot x_0$ of $f$ which satisfies the number theoretic condition from step 2 is highlighted.}
\end{figure}

\begin{proof}
In the style of the hard case of Lemma \ref{lem:3gentrans1}, our transition will be accomplished in three steps, with the first step altering the protocol for $(0,n)$, the second step using this altered protocol to make the desired transition for the $(1,b)$-edges, and the third step reversing the first.

To describe the first step, we will need to define a new protocol. First, in the context of a $\Z$-action, say in the orbit $\Z \cdot y_0$, the \textit{block protocol} $y_0$ for $n$ using the colors 1 and 2 will dictate that edges of the form $(a \cdot y_0 , (a+n) \cdot y_0)$ get the color 2 if $a \in [0,n)$ (mod $2n$) and the color 1 otherwise. In our original context, if $k \in \Z$ the \textit{offset protocol} $(x_0,k)$ for $(0,n)$ will dictate that the block protocols $x_0$ and $(1,k) \cdot x_0$ are followed on the $\Z$-orbits $\Z \cdot x_0$ and $\Z \cdot (1,0) \cdot x_0$ respectively.

Consider our starting sequence of protocols. In the proof of Lemma \ref{lem:1dbasictrans}, we saw that we could apply Lemma \ref{lem:general} with $\gamma_1 = (0,n)$ and $\gamma_2 = (1,m)$ by letting $P$ consist of pairs of $(0,n)$-orbits of the form $\{f,(1,m) \cdot f\}$ for $f \subset \Z \cdot x_0$. In that proof, we swapped parity along every pair, but here, let us only swap parity on those pairs for which $(0,2l) \cdot x_0 \in f$ for $l=0,\ldots,(n-1)/2$. The setup for this transition is shown in part (a) of Figure \ref{fig:uneven2}, and the result in part (b). The conclusion is that we can transition from the parallel protocol $x_0$ for $(0,n)$ and the two sided protocol $x_0$ for $(1,m)$ to the offset protocol $(x_0,m)$ for $(0,n)$ and the two sided protocol $x_0$ for $(1,m)$.

Now for step two, we claim that we can transition from the offset protocol $(x_0,m)$ for $(0,n)$ and the two sided protocol $x_0$ for $(1,b)$ to the offset protocol $(x_0,m)$ for $(0,n)$ and the two sided protocol $(1,0) \cdot x_0$ for $(1,b)$.

We want to apply Lemma \ref{lem:general} with $\gamma_1 = (1,b)$ and $\gamma_2 = (0,n)$ and the color 2 in place of the color 3. $P$ will be any partition of the $(1,b)$-orbits into pairs of the form $\{f, (0,n) \cdot f\}$, which is possible as $(0,n)$ has even order in the quotient $\Gamma / \langle (1,b) \rangle$ since $n$ is odd. One such pair is depicted in part (b) of Figure \ref{fig:uneven2}. Condition 3 from the lemma is clear by our use of a two sided protocol for $(1,b)$.

Condition 5 will be handled by an argument similar to that from the proof of Lemma \ref{lem:3gentrans2}. An $f$ as above can be described by a vertex of the form $(0,a_0) \cdot x_0$, so that the vertices in $f \cap \Z \cdot x_0$ are those of the form $(0,a_0 + 2kb) \cdot x_0$ for $k \in \Z$. Each such point determines a pair of parallel edges $e := ((0,a_0+2kb) \cdot x_0,(1,a_0+2kb+b) \cdot x_0)$ and $(0,n) \cdot e$ in $f$ and $(0,n) \cdot f$ respectively. If the two $(0,n)$-edges connecting $e$ and $(0,n) \cdot e$ both are given the color two by the offset protocol $(x_0,m)$, then this pair of edges will witness condition 5 from Lemma \ref{lem:general} as desired.

By definition of the offset protocol, this will be the case if both $a_0+2kb$ and $a_0+2kb+b-m$ are in the interval $[0,n)$ (mod $2n$), so we want to show that given $a_0$, we can always find a $k$ such that this is the case. Since $q = \gcd(n,2b)$, we can find a $k$ such that $a_0+2kb \in [0,q)$ (mod $2n$), and then since we are working under the assumption $b-m \in [0,n-q]$ (mod $2n$), we will also have $a_0+2kb+b-m \in [0,n)$ (mod $2n$) as desired. 
\end{proof}

\subsubsection{Generators from $2\Z$}

The only generators left to consider are those of the form $(0,b)$ for $b$ even. In our treatment of these, we will often work in one $\Z$-orbit at a time. We therefore start by establishing a key lemma in the setting of $\Z$-actions. For this, we start with a purely combinatorial game.

Suppose we have a finite graph consisting of a single path $v_0,\ldots,v_l$ of length $l$, where $v_i$ and $v_j$ are adjacent if and only if $|i - j| = 1$. Suppose the vertices of this path are labeled with two colors, say 5 and 6. In one \textit{move}, we are allowed to pick two adjacent vertices of the same color, and change both of their colors. Some examples of moves are shown in Figure \ref{fig:game}. 

\begin{figure}
\centering
\includegraphics[width=0.4\textwidth]{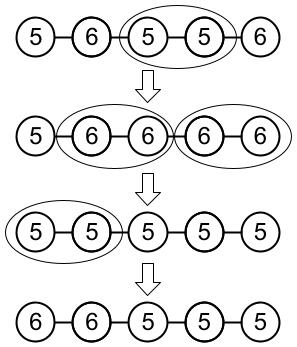}
\caption{\label{fig:game} An example of a sequence of moves in our game with $l = 4$. When a pair of adjacent vertices is circled, it indicates that our next move will be changing the color on those vertices. We sometimes draw two moves as occuring in a single step if the pairs of vertices they involve are disjoint. Note also that this sequence provides an example for Lemma \ref{lem:game}.}
\end{figure}

\begin{lemma}\label{lem:game}
If either $l$ is even and $v_0$ and $v_l$ have different colors, or $l$ is odd and $v_0$ and $v_l$ have the same color, then there is a sequence of moves which result in $v_0$ and $v_l$ switching their colors, but all other vertices keeping the same color.
\end{lemma}

\begin{proof}
We proceed by strong induction on $l$. The base case $l=1$ is clear. 

Now suppose $l > 1$. By hypothesis, there must be some pair $v_i,v_{i+1}$ of adjacent vertices with the same color. Pick the pair with $i$ minimal. Our first move will be changing the colors on this pair.

If $i = 0$, then $v_0$ is now the color we want it to be. If $i > 0$, then we now want to change the colors on $v_0$ and $v_i$, but leave the vertices between them the same color. By the minimality of $i$, in our initial coloring, $v_0$ and $v_i$ had the same color if and only if $i$ is even. Thus, after our first move, the path from $v_0$ to $v_i$ satisfies the hypotheses of the lemma, so by the inductive hypothesis,  we can indeed change the colors of $v_0$ and $v_i$ and leave the vertices between them the same color.

Essentially the same thing happens for the path from $v_{i+1}$ to $v_l$. We would like to check that in our initial coloring, $v_{i+1}$ and $v_l$ had the same color if and only if $l - (i+1)$ is even. We know as in the previous paragraph that $v_{i+1}$ has the same color as $v_0$ if and only if $i+1$ is odd. Also by hypothesis, $v_l$ has the same color as $v_0$ if and only if $l$ is odd. Combining these gives us what we want.
\end{proof}

We now turn to the promised $\Z$-action setting. Let $X'$ be an orbit of some $\Z$-action $a'$. Let $n',b' > 0$ with $n'$ odd, $b'$ even, and furthermore assume $\gcd(n',b') = 1$. Consider edge colorings of the graph $G(a',\{\pm n', \pm b' \}) \res X'$. 

Now a 2 coloring of the $b'$-edges, say with the colors 5 and 6, is completely determined by the colors it gives to the edges $(k \cdot x_0, (k+b') \cdot x_0)$ for $k=0,\ldots,b'-1$, where here $x_0 \in X'$ is some base point as before. Consider this sequence of colors, say $c(0),\ldots,c(b'-1)$, as a map $c: \Z/b' \rightarrow \{5,6\}$. Call this $c$ a $b'$-\textit{code}. We will say this coloring follows the \textit{one code protocol} $(x_0,c)$ for $b'$ using 5 and 6. 

For $\alpha \in \{5,6\}$, let $\neg \alpha$ denote the element of $\{5,6\}$ not equal to $\alpha$. Define the map $\overline{c}:\Z/(2b') \rightarrow \{5,6\}$ by $\overline{c}(k) = c(k)$ if $k \in [0,b')$ (mod $2b'$) and $\overline{c}(k) = \neg c(k)$ otherwise. (Note that we are implicitly sending $k$ through the natural quotient map $\Z/(2b') \rightarrow \Z/b'$. Similar implicit reductions are made throughout the remainder of the section.) $\overline{c}$ will be called the \textit{double code} of $c$. Of course, it carries the same information as $c$, but will sometimes be easier to work with. The point is that, in a coloring following one code protocol $(x_0,c)$ for $b'$, $\overline{c}(k)$ is the color received by the edge $(k \cdot x_0, (k+b') \cdot x_0)$ for each $k \in \Z$.

We are interested in using the $n'$-edges to transition between protocols for the $b'$ edges. Because of Proposition \ref{prop:Zexample}, we should not expect arbitrary transitions to be possible. We do however still have:

\begin{lemma}\label{lem:Ztransition}
In the above setting, if the $b'$-codes $c,c':\Z/b' \rightarrow \{5,6\}$ differ in an even number of entries, we can transition from the odd protocol $x_0$ for $n'$ and the one code protocol $(x_0,c)$ for $b'$ to the odd protocol $x_0$ for $n'$ and the one code protocol $(x_0,c')$ for $b'$.
\end{lemma}

We first note the following special case of the lemma:

\begin{lemma}\label{lem:1stepZ}
Suppose $k \in \Z/(2b')$ is such that $\overline{c}(k) = \overline{c}(k+n')$. Let $c' = \neg c$ at $k$ and $k+n'$, and $c' = c$ everywhere else. Then Lemma \ref{lem:Ztransition} holds for these values of $c$ and $c'$.
\end{lemma}

\begin{proof}
Let $f$ be the $b'$-orbit of $k \cdot x_0$. We want to apply Lemma \ref{lem:general} to change parity for $f$ and $n' \cdot f$. Observe that since $b'$ is even and $n'$ is odd and we use an odd protocol for the $n'$-edges, All edges of the form $(x,n' \cdot x)$ for $x \in f$ have the same color. This will give condition 5 of the lemma. Our hypothesis on the double code gives condition 3.
\end{proof}

\begin{figure}
\centering
\includegraphics[width=0.8\textwidth]{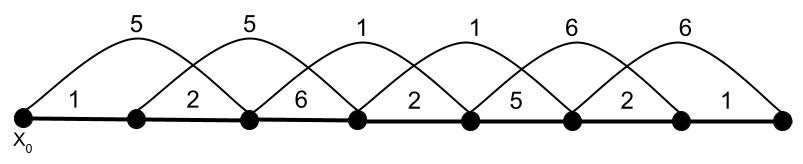}
\caption{\label{fig:1stepZ} A demonstration of the transition described in Lemma \ref{lem:1stepZ} with $n' = 1$, $b' = 2$. Note that on the left side of the figure, we follow one code protocol $(x_0,\{0 \mapsto 5, 1 \mapsto 5 \} )$ for $b'$, while on the the right, we follow one code protocol $(x_0,\{0 \mapsto 6, 1 \mapsto 6 \} )$ for $b'$. Here the $k$ from the lemma would be $0$.}
\end{figure}

Figure \ref{fig:1stepZ} shows the technique used in Lemma \ref{lem:1stepZ}. Now we can establish Lemma \ref{lem:Ztransition}:

\begin{proof}
Of course it suffices to consider the case where $c$ and $c'$ differ in exactly two entries, say at $r_1,r_2 \in \Z/b'$.

We first find a sequence $k_0,\ldots,k_l \in \Z/(2b')$ with the following properties:
\begin{enumerate}
    \item The $k_i$'s are distinct mod $b'$.
    \item $\{k_0,k_l\} = \{r_1,r_2\}$ (mod $b'$).
    \item $k_{i+1} = k_i + n'$ (mod $2b'$) for each $i$.
    \item $\overline{c}(k_0) = \overline{c}(k_l)$ if and only if $l$ is odd.
\end{enumerate}
To do so, first consider the sequence $j_0,\ldots,j_{b'} \in \Z/(2b')$ defined by $j_i = r_1 + in'$. Since $n'$ is coprime to $b'$, there must be some $0 < i_0 < b'$ such that $j_{i_0} = r_2$ (mod $b'$). Our sequence of $k_i$'s will either be the sequence $j_0,\ldots,j_{i_0}$ or the sequence $j_{i_0},\ldots,j_{b'}$. In either case, conditions $2$ and $3$ will hold by construction, and condition 1 will hold since $n'$ is coprime to $b'$.

For condition 4, we use the fact that $j_{b'} = j_0 + b'$ (mod $2b'$), and so $\overline{c}(j_{b'}) = \neg \overline{c}(j_0)$. Thus, whatever the value of $\overline{c}(j_{i_0})$, one of our two options above will give $\overline{c}(k_0) = \overline{c}(k_l)$ and the other will give $\overline{c}(k_0) = \neg \overline{c}(k_l)$. Since $b'$ is even, $i_0$ and $b'-i_0$ have the same parity, so one of these options will give us condition 4. 

Now, by Lemma \ref{lem:1stepZ}, we can make a transition to change the entries $c(k_i)$ and $c(k_{i+1})$ for some adjacent pair of elements $k_i$ and $k_{i+1}$ in our sequence if they are assigned the same value by our double code. This is exactly the type of move we are allowed to make in the game from Lemma \ref{lem:game}, though. Condition 4 is exactly the hypothesis from that lemma, so we can find a sequence of transitions which, cumulatively, change the code entry at $k_0$ and $k_l$, but nowhere else. By condition 2, this is exactly what we wanted.
\end{proof}

Before leaving the $\Z$-action setting, we define a generalization of the one code protocol which will be used below. Let $b''$ be a divisor of $b'$ such that $b'/b''$ is odd. Let $c_0:\Z/b'' \rightarrow \{5,6\}$ be a $b''$-code. The \textit{short code protocol} $(x_0,c_0)$ for $b'$ will be the one code protocol $(x_0,c)$ for $b'$, where $c$ is the $b'$-code given by $c(k) = c_0(k)$ for $k \in [0,b'') \cup [2b'',3b'') \cup \cdots [b'-b'',b')$ and $c(k) = \neg c_0(k)$ elsewhere. The point of requiring $b'/b''$ to be odd is that it implies $\overline{c}(k) = \overline{c_0}(k)$ for each $k \in \Z/(2b')$.

We now return to our original situation with $\Gamma = \Z/2 \times \Z$. Fix $b$ even. Fix $(0,n)$ and $(1,m)$ as in previous sections (so $n$ is odd). Let $q = \gcd(b,n)$, $b' = b/q$ and $n' = n/q$. Also let $b''$ be the largest power of 2 dividing $b'$ (equivalently $b$).

If $b'' \mid m$, then by working separately on each $\Z/2 \times b''\Z$-orbit, we can treat the $(0,b)$ as $(0,b/b'')$. Since $b/b''$ is odd, Lemma \ref{lem:1dbasictrans} will turn out to be enough to handle this. See Lemma \ref{lem:specialeventrans} later. Thus assume $b'' \not|\ m$.




Let $c,d :\Z/b'' \rightarrow \{5,6\}$ be $b''$-codes. The \textit{two code protocol} $(x_0,c,d)$ for $(0,b)$ will dictate that the short code protocols $(x_0,c)$ and $((1,0) \cdot x_0, d)$ for $b$ are followed on the $\Z$-orbits $\Z \cdot x_0$ and $\Z \cdot (1,0) \cdot x_0$ respectively. Note that $b/b''$ is indeed odd.

Let $c_0:\Z/b'' \rightarrow \{5,6\}$ be the $b''$-code taking the constant value 5. An example of the two code protocol $(x_0,c_0,c_0)$ for $(0,b)$ is shown in part (a) of Figure \ref{fig:even}. We would like to transition between the basepoints $x_0$ and $(0,1) \cdot x_0$ for such protocols. Note that the protocol $((0,1) \cdot x_0,c_0,c_0)$ for $(0,b)$ is the same as the protocol $(x_0,c_0',c_0')$, where $c_0'$ is the $b''$-code given by $c_0'(0) = 6$ but $c_0' = 5$ everywhere else.

\begin{lemma}\label{lem:eventrans}
We can transition from the parallel protocol $x_0$ for $(0,n)$, the two sided protocol $x_0$ for $(1,m)$, and the two code protocol $(x_0,c_0,c_0)$ for $(0,b)$ to the parallel protocol $x_0$ for $(0,n)$, the two sided protocol $x_0$ for $(1,m)$, and the two code protocol $(x_0,c_0',c_0')$ for $(0,b)$.

Likewise if the parallel protocols for $(0,n)$ are replaced with alternating protocols.
\end{lemma}

\begin{figure}
\centering
\includegraphics[width=0.8\textwidth]{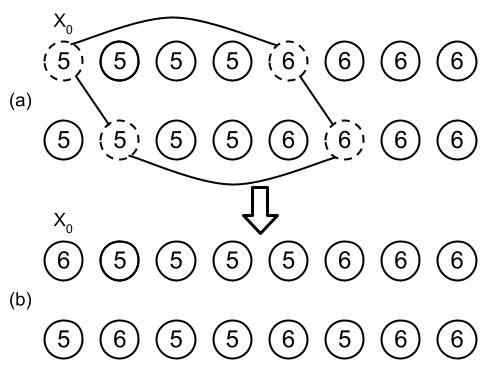}
\caption{\label{fig:even} A demonstration of step 1 in the procedure in the proof of lemma \ref{lem:eventrans} when $b=b'=b''=4$ and $m=1$, in the style of Figure \ref{fig:uneven2}. As in that figure, the label of a vertex $x$ denotes the color of the edge $(x,(0,b) \cdot x)$. (a): A coloring following two code protocol $(x_0,c_0,c_0)$ for $(0,b)$, along with the setup for step 1. The dashed vertices are those included in one of the $(0,b)$-orbits along which parity is changed. The edges drawn are the $(0,b)$ and $(1,m)$-edges involved in making that change for these orbits. (b): The effects of step 1. Observe that in the orbit $\Z \cdot x_0$, the $b$-edges now follow short code protocol $(1 \cdot x_0,c_0) = (x_0,c_0')$, while in the other $\Z$-orbit, they follow short code protocol $((1,0) \cdot x_0,c_0'')$.}
\end{figure}

\begin{proof}
Recall we have assumed $b'' \not|\ m$. Further assume $m \in (0,b'')$ (mod $2b''$), as the argument in the case $m \in (b'',2b'')$ (mod $2b''$) will be similar.

Our procedure will involve two steps. In step 1, we apply Lemma \ref{lem:general} with $\gamma_1 = (0,b)$ and $\gamma_2 = (1,m)$ to swap parity along pairs $f$, $(1,m) \cdot f$ of $(0,b)$-orbits for which $f$ contains a point of the form $(0,kb'') \cdot x_0$. Condition 5 follows from our use of a two sided protocol for $(1,m)$. Condition 3 for these pairs follows from our assumption $m \in (0,b'')$ (mod $2b''$). The set up for this step is shown in part (a) of Figure \ref{fig:even} and the result in part (b). The conclusion is that we can transition from the two sided protocol $x_0$ for $(1,m)$ and the two code protocol $(x_0,c_0,c_0)$ for $(0,b)$ to the two sided protocol $x_0$ for $(1,m)$ and the two code protocol $(x_0,c_0',c_0'')$ for $(0,b)$, where $c_0''$ is the $b''$-code given by $c_0''(m) = 6$ but $c_0'' = 5$ everywhere else. Note for later that $c_0'$ and $c_0''$ differ in two places: at 0 and at $m$.

Thus, after this step, the short code protocol for $b$ followed on the orbit $\Z \cdot x_0$ is the correct one. In step 2, then, we will work within the orbit $\Z \cdot (1,0) \cdot x_0$ to correct things. We need to show that if $y_0$ is a point in a $\Z$-orbit, we can transition from the odd protocol $y_0$ for $n$ and the short code protocol $c_0''$ for $b$ to the odd protocol $y_0$ for $n$ and the short code protocol $c_0'$ for $b$. (So in our application, $y_0$ is $(1,0) \cdot x_0$). 

For this, we will work separately within each $q\Z$-orbit. Fix one, say $q\Z \cdot z_0$, where $z_0 = r \cdot y_0$ for some $0 \leq r < q$. Within this orbit, the generators $b$ and $n$ behave like $b'$ and $n'$, and these satisfy the hypotheses used for the $b'$ and $n'$ from the setting of Lemma \ref{lem:Ztransition}. 

Under this identification, since $q$ is odd, our odd protocol for $n$ turns into an odd protocol for $n'$. The short code protocol $(y_0,c_0'')$ for $b$ turns into the one code protocol $(z_0,c'')$ for $b'$, where $c''$ is the $b'$-code whose double code is given by $\overline{c''}(k) = \overline{c_0''}(r+kq)$ for $k \in \Z/(2b')$. Likewise, the short code protocol $(y_0,c_0')$ for $b$ turns into the one code protocol $(z_0,c')$, where $c'$ is the $b'$-code defined from $c_0'$ in the same way $c''$ is from $c_0''$.

So, by Lemma \ref{lem:Ztransition}, it suffices to show that $c'$ and $c''$ differ in an even number of entries. Now $q$ is coprime to $2b''$ since the former is odd and the latter is a power of 2. Therefore, as $k$ ranges over $\Z/(2b')$, $r+kq$ will take every value mod $2b''$ $b'/b''$ times. Therefore, since $\overline{c_0'}$ and $\overline{c_0''}$ differ in 4 places, $\overline{c'}$ and $\overline{c''}$ differ in $4b'/b''$ places, and so $c'$ and $c''$ differ in $2b'/b''$ places. This is indeed even.
\end{proof}

Finally, we address the case $b'' \mid m$ mentioned earlier. Let $d_0$ be the $b''$-code taking constant value 6.

\begin{lemma}\label{lem:specialeventrans}
If $m/b''$ is even, we can transition from the two sided protocol $x_0$ for $(1,m)$ and the two code protocol $(x_0,c_0,c_0)$ for $(0,b)$ to the two sided protocol $x_0$ for $(1,m)$ and the two code protocol $((0,1) \cdot x_0,c_0,c_0)$ for $(0,b)$.

If $m/b''$ is odd, we can transition from the two sided protocol $x_0$ for $(1,m)$ and the two code protocol $(x_0,c_0,d_0)$ for $(0,b)$ to the two sided protocol $x_0$ for $(1,m)$ and the two code protocol $((0,1) \cdot x_0,c_0,d_0)$ for $(0,b)$.
\end{lemma}

\begin{proof}
We work separately in each $(\Z/2 \times b''\Z)$-orbit. Observe that in every orbit other than the one containing $x_0$, no change needs to occur.

In that orbit, identifying $\Z/2 \times b''\Z$ with $\Z/2 \times \Z$ by identifying $b''$ with 1, the two sided protocol $x_0$ for $(1,m)$ turns into the two sided protocol $x_0$ for $(1,m/b'')$, the two code protocol $(x_0,c_0,c_0)$ for $(0,b)$ turns into the parallel protocol $x_0$ for $(0,b/b'')$ (note $b/b''$ is odd), and the two code protocol $((0,1) \cdot x_0,c_0,c_0)$ for $(0,b)$ turns into the parallel protocol $(0,1) \cdot x_0$ for $(0,b/b'')$. Likewise for the two code protocols using $c_0$ and $d_0$ if ``parallel'' is replaced by ``alternating''. In all cases, we are done by Lemma \ref{lem:1dbasictrans}.
\end{proof}

Note that being able to transition to the basepoint $(0,1) \cdot x_0$ here is enough to be able to transition between any pair of basepoints for our two-code protocols since the two code protocol $((1,0) \cdot x_0, c_0,c_0)$ for $(0,b)$ is the same as the two code protocol $(x_0,c_0,c_0)$ for $(0,b)$, and likewise for $((1,0) \cdot x_0,c_0,d_0)$ and $((0,b'') \cdot x_0,c_0,d_0)$.

\subsubsection{Combining the cases}

Recall that at the beginning of subsection \ref{subsec:1}, we defined a set $S' \subset S$ of generators, and assumed $(1,0) \not\in S$. In this final subsubsection, we give a way of choosing protocols for all the elements of $S'$ so that, by the work of the previous subsubsections, we can transition between any two sequences of basepoints for these protocols. Thus we will conclude:

\begin{lemma}\label{lem:no(1,0)}
If $(1,0) \notin S$, $\chi_B'(\Gamma,S) = |S|$.
\end{lemma}

Of course, Lemmas \ref{lem:(1,0)} and \ref{lem:no(1,0)} together complete our proof of statement 2 of Theorem \ref{th:Z^d}, and hence of Theorem \ref{th:Z^d}.

Our proof will be a little messier than the corresponding proof at the end of subsection \ref{subsec:2} thanks to the lack of a suitable analog of the $\sim$ relation used there.

Let us begin. There must be some element $(\epsilon,n) \in S'$ for which $n$ is odd. Note that there is an automorphism of $\Gamma$ exchanging $(1,1)$ and $(0,1)$. This will send $(\epsilon,n)$ to $(1-\epsilon,n)$, so we may assume $S'$ has an element of the form $(0,n)$ for $n$ odd. Then it must also have an element of the form $(1,m)$. Fix such elements for the rest of the argument.

We will use two sided protocols for $(1,m)$, and for $(0,n)$ we will use parallel protocols if $m$ is even and alternating protocols if it is odd.
By Lemma \ref{lem:1dbasictrans}, we will be able to transition between any two pairs of basepoints here.

For other elements of $S'$ of the form $(0,n')$ for $n'$ odd, we will also use the parallel protocols if $m$ is even and alternating protocols if it is odd. Again by the lemma, we can use the $(1,m)$-edges to transition between different basepoints for $(0,n')$.

Similarly, we will use two sided protocols for $(1,m')$ for all elements of that form in $S'$. If for such an element, $m' = m$ (mod 2), then again by the lemma we can use the $(0,n)$-edges to transition between different basepoints for $(1,m')$.

If $S'$ has elements of the form $(1,m')$ for both odd and even $m$, let us choose our fixed $m$ to be even (so that we use parallel protocols for $(0,n)$). Then in the previous paragraph we have taken care of all generators of the form $(1,m')$ for $m'$ even. If $(1,b) \in S'$ with $b$ odd, then Lemma \ref{lem:diffparity} tells us we can use the $(0,n)$ and $(1,m)$ edges to transition between different basepoints for $(1,b)$.

It remains to deal with elements of $S'$ of the form $(0,b)$ for $b$ even. Fix such a $b$. As in the previous subsubsection, let $b''$ be the largest power of 2 dividing $b$, $c_0$ be the $b''$-code taking constant value 5, and $d_0$ the $b''$-code taking constant value 6.

If $b'' \mid m$, we will use for $(0,b)$ two code protocols of the form $(x_0,c_0,c_0)$ or $(x_0,c_0,d_0)$ according to the parity of $m/b''$ as prescribed by Lemma \ref{lem:specialeventrans}. By that lemma and the comment after its proof, we can use the $(1,m)$-edges to transition between different basepoints for $(0,b)$.

If $b'' \not|\ m$, we will use for $(0,b)$ two code protocols of the form $(x_0,c_0,c_0)$. By Lemma \ref{lem:eventrans}, we can use the $(0,n)$ and $(1,m)$ edges together to transition between different basepoints for $(0,b)$.

\section*{Acknowledgements}

The author was partially supported by the ARCS foundation, Pittsburgh chapter. He also thanks Clinton Conley for many helpful discussions. He is most indebted to the anonymous referee, whose thoughtful feedback helped greatly improve the presentation of this paper.

\end{document}